\documentclass[11pt]{article}%
\usepackage{amssymb}
\usepackage[all]{xy}
\usepackage{amsfonts}
\usepackage{amsmath}
\usepackage{amssymb}
\usepackage{amsthm}
\usepackage{url}
\usepackage{hyperref}%
\usepackage{graphics}
\usepackage{graphicx}
\usepackage{algorithm}
\usepackage{algpseudocode}

\usepackage{epsfig}
\usepackage{subfigure}
\usepackage[all]{xy}
\usepackage{color}

\def\E{\mathbb{E}}

\def\R{\mathbb{R}}
\def\P{\mathcal{P}}

\newtheorem{thm}{Theorem}[section]
\newtheorem{ass}{Assumption}

\newtheorem{lem}[thm]{Lemma}

\newtheorem{prop}[thm]{Proposition}

\newtheorem{rem}{Remark}[section]
\numberwithin{equation}{section}
\begin{document}
\title{Unbiased Monte Carlo estimate of stochastic differential equations expectations}
\author{Mahamadou Doumbia
  \and
  Nadia Oudjane \thanks{EDF R\&D \& FiME, Laboratoire de Finance des March\'es de l'Energie (www.fime-lab.org)}
  \and
  Xavier  Warin  \footnotemark[1] }
\maketitle

\begin{abstract}
  We propose an unbiased Monte Carlo method to compute $\E(g(X_T))$ where $g$ is  a  Lipschitz function and $X$  an Ito process.
  This approach extends the method proposed in \cite{HTT} to the case where $X$ is solution of a multidimensional stochastic differential equation with varying drift and diffusion coefficients. A variance reduction method relying on interacting particle systems is also developed. 
\vspace{5mm}


\end{abstract}

\section{Introduction}
%-------------------------------------------
Let $d\ge 1$ and $W$ be a $d-$dimensional Brownian motion.
We introduce the process $X$ defined as  the unique strong solution of the multi-dimensional Stochastic differential Equation (SDE) with coefficients satisfying the usual Lipschitz conditions  :
\begin{equation}
  \label{eq:sde}
  \left \{
  \begin{array}{lll}
    d X_t^{0,x_0} & =& b(t,X_t^{0,x_0}) dt + \sigma(t,X_t^{0,x_0}) dW_t, \\
    X_0^{0,x_0}  & =& x_0,
  \end{array}
  \right.
\end{equation}
where $b :[0,T] \times \R^d \longrightarrow \R^d$ is the drift and $\sigma : [0,T] \times  \R^d \longrightarrow {\cal S}^d$   is the diffusion  of the process, ${\cal S}^d$ being the set of $d \times d$ dimensional matrices.

In this paper, we are interested in a Monte Carlo approach to compute an expectation of the form
\begin{equation}
\label{eq:expectation}
u(t,x):=\E[g(X_T^{t,x})]\ .
\end{equation}
When no explicit solution is available, the classical method to solve equation \eqref{eq:expectation} consists in  using a  discretization scheme of \eqref{eq:sde}  (for example the Euler scheme~\cite{kloeden}, the Milstein scheme~\cite{milstein}, or the Burrage scheme~\cite{burrage}) and the error can be decomposed as a sum of an error due to the discretization time step $\delta t$ and a statistical error of order $N^{-1/2}$ due to the Monte Carlo method for a number $N$ of simulations.\\
%Faiblesses des approches usuelles
In principle, this  bias/variance tradeoff should carefully be adjusted in order to optimize the rate of convergence. This type of analysis has been conducted in~\cite{Duffie} showing that, for instance with the the simple \textit{Euler Monte Carlo} method, (using the Euler scheme to discretize the time), the best choice of time step $\delta t$ as a function of the sample size $N$ would lead to a rate of order $c_N^{-\frac{1}{3}}$, where $c_N=N/\delta t$ measures the computing time. Hence the combination of the bias and variance error deteriorates the standard rate $c_N^{-1/2}$, due to the statistical error, when $X$ is easily simulatable and $c_N=N$. Moreover, in practice it is difficult to evaluate properly the bias error so that the optimal tradeoff is rarely practicable. \\
The \textit{Multilevel Monte Carlo} (MLMC) method introduced in~\cite{Giles} is a way to improve the bias/variance tradeoff and to reduce the variance by combining several \textit{Euler-Monte Carlo} estimates, associated with different time discretization steps.  The idea is then to adjust judiciously the size of the sample simulated for each discretization \textit{level}, in order to achieve a better rate of convergence. \\
This approach has been extended in~\cite{Rhee} allowing for an infinite number of levels so that the bias vanishes. The estimate is then expressed as an infinite sum (over the levels), which is randomized by introducing a probability distribution driving the levels.  
However, when the order of the time discretization scheme is not sufficiently high, this method results in an infinite variance estimate. More precisely,  as soon as the time discretization scheme implies a strong error greater than or equal to the order $\sqrt{\delta t}$, either the variance or the computing time blows up. Unfortunately this situation includes the case of the Euler scheme, which is so far the most widely usable discretization scheme in multidimensional cases. \\
This approach has been improved in~\cite{Andersson}, where the authors rely on the parametrix expansion presented in~\cite{Bally} to propose a finite variance estimate. More specifically, the parametrix method provides a precise expansion of the expected difference considered at two successive levels in terms of a difference between the infinitesimal generator, $\mathcal{L}$, associated with~\eqref{eq:sde} and the one associated with the same SDE with frozen coefficients at a given point, as defined hereafter by~\eqref{eq:tildeL}. Finally, importance sampling is used to change the levels distribution in order to control the variance.  
These developments lead to the \textit{backward simulation method} or the \textit{forward simulation method}, depending on whether  $\mathcal{L}$ or its adjoint is used to represent the expectation. 
The backward method consists in generating some independent and identically distributed (i.i.d.) Euler type discretizations of a process, at random discrete times, from time $T$ to time $t$. The payoff function, $g$ is used as initial distribution at time $T$ and the estimator results from a weighted average over the trials that hit the initial point, $x$, at time $t$. Therefore, this approach requires the payoff $g$ to be integrable and is limited to small dimensions (for which the probability of reaching a given point can efficiently be computed).
The forward method consists in generating some i.i.d. Euler type discretizations of~\eqref{eq:sde} at random discrete times from time $t$ to $T$ and then computing a weighted average of the payoff function evaluated at the final points. In both methods the weights depend on the drift and diffusion coefficients $b$ and $\sigma$ evaluated along the simulated path. However, the forward approach relies on a stronger regularity assumption on the SDE coefficients. In particular, the related weights involve the first derivatives of the drift and the first and second derivatives of the volatility function.  \\
Another approach called \textit{Exact simulation} was initialized in~\cite{Beskos}. The idea relies on Lamperti transform to come down to a unit diffusion process. It has been extended to more general SDEs in for instance~\cite{Chen,Jourdain}. However, the Lamperti transform is limited to the one dimensional case and extensions to the multidimensional are still limited to some specific cases. 

%%%%%%%%%%%%%%%%%%%%%%%%%%%%%%%%%%%%%%%%%%%%%%%%%%%%
In this paper, we propose to extend  a method originally developed in \cite{HTT}.
The main idea developed in this seminal paper is to start by simulating exactly a SDE :
%
%\begin{eqnarray*}
  %\left \{
  %\begin{array}{lll}
    %d \hat X_t^{0,x_0} & =& \hat b(t,\hat X_t^{0,x_0}) dt + \hat \sigma(t,\hat X_t^{0,x_0}) dW_t \\
    %\hat X_0^{0,x_0}  & = &x_0\ ,
  %\end{array}
  %\right.
%\end{eqnarray*}
\begin{eqnarray*}
  \left \{
  \begin{array}{lll}
    d Y_t^{0,x_0} & =& \hat b(t,Y_t^{0,x_0}) dt + \hat \sigma(t,Y_t^{0,x_0}) dW_t \\
    Y_0^{0,x_0}  & = &x_0\ ,
  \end{array}
  \right.
\end{eqnarray*}
where the coefficients $\hat b$ and $\hat \sigma$ are updated at independent exponential switching times.
Then the change in coefficients in SDE \eqref{eq:sde} is taken into account in an expectation representation via weights derived from the automatic differentiation technique developed in \cite{fournie}. 
By carefully choosing the coefficients $\hat b$, $\hat \sigma$, the authors were able to provide a finite variance method in the case where the diffusion coefficient is constant or with a general diffusion term but without drift and in dimension one. 
However, the variance of the resulting estimator is proved to be infinite in the most general case.
One interest of this approach which is very similar to the forward parametrix representation~\cite{Bally, Andersson} is that the weights do not involve any derivatives of the coefficients $b$ or $\sigma$ so that no differentiability assumptions on those coefficients is required. 
Besides, one major motivation of this type of approach goes beyond the scope of the present paper. The idea is to generalize the branching diffusion representation of  nonlinear Partial Differential Equations (PDEs)  considered in~\cite{labordere, HTTbsde} to a more general class of nonlineariries. One step in that direction has already been done in~\cite{HOTTW} with an extension of the branching diffusion representation to a class of semilinear PDEs. \\ 
To bypass the infinite variance obstacle faced in~\cite{HTT}, the idea developed in the present paper consists in extending the original framework to more general switching times and exploit the switching time distribution to control the estimator variance. Notice that the same idea has been independently investigated in~\cite{Andersson} to control the variance of the parametrix representation proposed in~\cite{Bally}. 
We prove that under suitable assumptions on the switching times distribution, we can provide a finite variance estimate of the solution of~\eqref{eq:expectation} in the most  general case with drift and diffusion coefficients both varying. For instance, the gamma distribution is proved to verify those assumptions as soon as the shape parameter $\kappa$ satisfies $\kappa \le \alpha \wedge \frac{1}{2}$, when the coefficients $b$ and $a:=\sigma\sigma^\top$ are supposed to be uniformly $\alpha$-H\"older continuous w.r.t the time variable. 
Another contribution consists in proposing an original interacting particle scheme that helps to stabilize even more the estimator. This approach results in a new estimator combining both branching  and interacting particle techniques. The new estimator is proved to be unbiased with finite variance. Finally, numerical tests confirm the interest of our new algorithm showing significant variance reduction in various examples.    

\section{Notations}
%-------------------------------------------
Let  $C^{1,2}_b([0,T]\times \R^d,\R)$ denote the set of continuously differentiable bounded functions with bounded derivatives of order $1$ for the time variable and bounded derivatives up to order $2$ for the space variable. Let $\mathcal{L}$ denote the infinitesimal generator associated with~\eqref{eq:sde} such that for any sufficiently regular function $\varphi\,:\, [0,T]\times \R^d \mapsto \R$ in the domain of $\mathcal{L}$, $\mathcal{L}\varphi$ is given as the real valued function such that   
\begin{equation}
\label{eq:L}
( \mathcal{L}\varphi )(t,x)=b(t,x).D\varphi (t,x)+\frac{1}{2} a(t,x): D^2 \varphi (t,x)  \ , \quad \textrm{for all}\ (t,x)\in [0,T]\times \R^d\ ,
\end{equation}
where $a(t,x):= \sigma(t,x)\sigma(t,x)^\top$,  $A:B := tr(A B^\top)$ and $D$ (resp. $D^2$) denotes the differential operator of order $1$ (resp. of order 2) w.r.t. the space variable $x$.  
Let us consider a real valued Lipschitz continuous function $g$ defined on $\R^d$. 
By the Feynman-Kac formula it is well-known that if there exists $v^\ast\in C^{1,2}_b([0,T]\times \R^d,\R)$ solution of the linear Partial Differential Equation (PDE)   
\begin{equation}
\label{eq:PDE1}
\left\{
\begin{array}{l}
\partial_t v+\mathcal{L} v=0\\
v(T,x)=g(x)\ ,
\end{array}
\right .
\end{equation}
then this PDE has a unique classical solution $v^\ast(t,x)=u(t,x)=\E[g(X_T^{t,x})]$.
In the sequel $\Vert x\Vert $ stands for the $L_\infty$ norm of a vector or a matrix $x$.\\
First we introduce an intermediary assumption that will be relaxed for our main results:
\begin{ass}
\label{ass:unique1}
 The linear PDE~\eqref{eq:PDE1} admits a unique classical solution $v^\ast\in C^{1,2}_b$. 
\end{ass}
All along this paper,  the following assumption will be in force.
\begin{ass}
\label{ass:unique2}
\begin{enumerate}
\item The diffusion  $\sigma (t,x)$ is non-degenerated such that for some constant $\epsilon_0 >0$:
  $$  a(t,x) \ge \epsilon_0 \mathbb{I}, \quad \forall (t,x)\in [0,T]\times \R^d. $$
\item $b$ and $a$ are uniformly Lipschitz w.r.t. the space variable i.e. there exists a finite constant $L$ such that for any $(t,x,x')\in [0,T]\times \R^d\times \R^d$ 
$$
\| b(t,x)-b(t,x')\| +\| a(t,x)-a(t,x')\| \leq L\| x-x'\|\ .
$$
	\item There exists $\alpha \in (0,1]$ such that $b$ and $a$ are uniformly $\alpha$-H\"older continuous w.r.t. variable $t$ i.e. there exists  a finite constant $H$ such that for any $(t,t',x)\in [0,T]\times [0,T]\times \R^d$
$$
\| b(t,x)-b(t',x)\| +\| a(t,x)-a(t',x)\| \leq H\vert t-t'\vert^{\alpha}\ .
$$
\end{enumerate}
\end{ass}
For a fixed point $(\tilde t,\tilde x)\in [0,T]\times \R^d$, we introduce some operators and processes that will be useful in the sequel  
\begin{itemize}
	\item $ \mathcal{L}^{\tilde t,\tilde x}$ the differential operator similar to $\mathcal{L}$ with the drift and diffusion  frozen at $(\tilde t,\tilde x)$ such that for any regular function $\varphi$ in the domain of $\mathcal{L}^{\tilde t,\tilde x}$ 
\begin{equation}
\label{eq:tildeL}
\mathcal{L}^{\tilde t,\tilde x}
\varphi (t,x)=
b(\tilde t,\tilde x).D\varphi (t,x)+\frac{1}{2} a(\tilde t,\tilde x) : D^2 \varphi (t,x) \ , \quad 
\textrm{for all}\ (t,x)\in [0,T]\times \R^d \ ,
\end{equation}
	\item $(\tilde X_t^{\tilde t,\tilde x, t_0,x_0})_{t\geq t_0}$ the Gaussian process with infinitesimal operator $\mathcal{L}^{\tilde t,\tilde x}$ defined by 
	\begin{equation}
	\label{eq:tildeX}
	\tilde X^{\tilde t,\tilde x,t_0,x_0}_t=x_0+b(\tilde t, \tilde x) (t-t_0)+\sigma(\tilde t, \tilde x) (W_t-W_{t_0})\ .
	\end{equation}
	for a given initial condition $(t_0,x_0)\in [0,T]\times \R^d$.
	\item $h^{\ast,\tilde t,\tilde x}\,:\, [0,T]\times \R^d \mapsto \R$ involving the unique solution $v^\ast$ of~\eqref{eq:PDE1} is defined by
	\begin{equation}
	\label{eq:h}
	h^{\ast,\tilde t,\tilde x}(t,x):=(b(t,x)-b( \tilde t, \tilde x) ).Dv^\ast(t,x)+\frac{1}{2}(a(t,x)-a(\tilde t, \tilde x)):D^2v^\ast (t,x)\ .
	\end{equation}
	Notice that $h^{\ast,\tilde t,\tilde x}$ is a well defined continuous function since $v^\ast \in C^{1,2}_b$ and in particular 
	\begin{equation}
	\label{eq:rqueh}
	h^{\ast,\tilde t,\tilde x}(t,x)=\mathcal{ L}v^\ast(t,x)-\mathcal{L}^{\tilde t, \tilde x}v^\ast(t,x)\quad \textrm{for all}\ (t,x)\in [0,T]\times \R^d \ .
	\end{equation}
\end{itemize}

\section{Probabilistic representation using a regime switching process}
%-------------------------------------------
Recalling~\cite{HTT}, the following representation  holds
\begin{lem}
\label{lem:Represent}
Suppose that  Assumptions ~\ref{ass:unique1} and~\ref{ass:unique2}  hold  and $\tilde X^{\tilde t,\tilde x}$ is the Gaussian process defined in~\eqref{eq:tildeX}, then $u$ defined by~\eqref{eq:expectation} and its (bounded and continuous) derivatives $Du$ and $D^2u$ are solutions of the system
\begin{equation}
\label{eq:Represent}
\begin{array}{lll}
u(t,x)&=&\E[g(\tilde X_T^{\tilde t, \tilde x,t,x})+\int_t^T H^{\tilde t,\tilde x}(s,\tilde X^{\tilde t,\tilde x, t ,x }_s,Du(s,\tilde X^{\tilde t,\tilde x, t, x}_s),D^2u(s,\tilde X^{\tilde t,\tilde x, t,x}_s))\,ds]\\
Du(t,x)&=&\E[g(\tilde X_T^{ \tilde t,\tilde x, t,x})\mathcal{M}^{\tilde t,\tilde x}_{t,T}+  \\
&  & \int_t^T H^{\tilde t,\tilde x}(s,\tilde X^{\tilde t,\tilde x,t,x}_s,Du(s,\tilde X^{\tilde t , \tilde x, t,x}_s),D^2u(s,\tilde X^{\tilde t , \tilde x, t,x}_s))\mathcal{M}^{\tilde t,\tilde x}_{t,s}\,ds]\\
D^2u(t,x)&=&\E[g(\tilde X_T^{  t,  x, t,x})\mathcal{V}^{ t, x}_{t,T}+  \\
& & \int_t^T H^{ t,x}(s,\tilde X^{ t,  x, t,x}_s,Du(s,\tilde X^{ t,  x, t,x}_s),D^2u(s,\tilde X^{ t,  x, t,x}_s))\mathcal{V}^{t, x}_{t,s}\,ds]\ ,
\end{array}
\end{equation}
where for any $(\tilde t,\tilde x)\in [0,T]\times \R^d$ the function $H^{\tilde t,\tilde x}:[0,T]\times\R^d\times \R^d\times {\cal S}^d	\mapsto \R$ is such that 
\begin{equation}
\label{eq:H}
H^{\tilde t,\tilde x}(t,x,y,z)=(b(t,x)-b( \tilde t,\tilde x) ).y+\frac{1}{2}(a(t,x)-a(\tilde t,\tilde  x)):z\ ,
\end{equation}
$\mathcal{M}^{\tilde t,\tilde x}_{t,s}$ and $\mathcal{V}^{\tilde t,\tilde x}_{t,s}$ are respectively the first and second order Malliavin weights associated with the process $\tilde X^{\tilde t,\tilde x}$ that is using $\delta_{t,s} W = W_s-W_t$
\begin{equation}
\label{eq:Malliavin}
\mathcal{M}^{\tilde t,\tilde x}_{t,s}:=(\sigma(\tilde t,\tilde x)^{-1})^\top\frac{\delta_{t,s} W }{s-t}\ ,\quad\textrm{and}\quad \mathcal{V}^{\tilde t,\tilde x}_{t,s}:=(\sigma (\tilde t,\tilde x)^{-1})^\top\frac{\delta_{t,s} W \delta_{t,s} W ^\top -(s-t) \mathbb{I}}{(s-t)^2} \sigma (\tilde t,\tilde x)^{-1} \ .
\end{equation}
\end{lem}
\begin{proof}
The proof relies on the uniqueness property of classical solutions of PDEs satisfying the Feynman-Kac representation. 
Notice that under Assumptions ~\ref{ass:unique1} ~\ref{ass:unique2}, $u$ is the unique classical solution of~\eqref{eq:PDE1}. Of course, thanks to equation~\eqref{eq:rqueh}, for any $(\tilde t,\tilde x)\in [0,T]\times \R^d$ $u$ is also a $C^{1,2}_b$ solution of the following linear PDE
$$
\partial_t u +\mathcal{ L}^{\tilde t,\tilde x} u+h^{\ast,\tilde t,\tilde x} = 0 \ .
$$
Then one can use again Feynman-Kac formula to represent the unique solution $u$ of the above PDE as
\begin{equation}
\label{eq:uFK}
u(t, x)=\E[g(\tilde X_T^{\tilde t, \tilde x, t , x})+\int_{t}^T h^{\ast,\tilde t,\tilde x}(s,\tilde X^{\tilde t,\tilde x, t, x}_s)\,ds]\ .
\end{equation}
Finally observe that 
\begin{equation}
\label{eq:hH}
h^{\ast,\tilde t,\tilde x}(s,\tilde X^{\tilde t,\tilde x, t, x}_s)=H^{\tilde t,\tilde x}(s,\tilde X^{\tilde t,\tilde x,t, x}_s,Du(s,\tilde X^{\tilde t,\tilde x}_s),D^2u(s,\tilde X^{\tilde t,\tilde x, t,x }_s))\ .
\end{equation}

The equations relative to $Dv$ and $D^2v$ are obtained by applying Elworthy's formula~\cite{fournier1} (which simply results here in the Likelihood ratio of Broadie and Glasserman~\cite{BG96}) in~\eqref{eq:uFK} and by using some technical estimates placed in the Appendix~\ref{sec:Technics} to be able to differentiate under the time integral. 

\end{proof}

Let $\tau$ be a random time independent of the Brownian $W$ following the density $f$ supposed to be strictly positive on $[0,\infty]$ and $\P[\tau >T] >0$. One can rewrite representation~\eqref{eq:Represent} by using a change of measure to replace the time integral by an expectation according to the random time $\tau$. 
\begin{eqnarray}
\label{eq:RepresentTau}
u(t,x)&=&\frac{\E[g(\tilde X_T^{\tilde t,\tilde x, t, x})\mathbf{1}_{\tau\geq T-t}]}{1-F(T-t)}\nonumber \\
&&+\E[ \frac{H^{\tilde t,\tilde x}(t+\tau,\tilde X^{\tilde t , \tilde x, t,x}_{t+\tau},Du(t+\tau,\tilde X^{\tilde t ,\tilde x, t,x}_{t+\tau}),D^2u(t+\tau,\tilde X^{\tilde t , \tilde x, t,x}_{t+\tau}))}{f(\tau)}\,\mathbf{1}_{\tau<T-t}]\nonumber \\
Du(t,x)&=&\frac{\E[g(\tilde X_T^{\tilde t,\tilde x, t, x})\mathcal{M}^{\tilde t,\tilde x}_{t,T}\mathbf{1}_{\tau\geq T-t}]}{1-F(T-t)} \nonumber \\
&&+\E[ \frac{H^{\tilde t,\tilde x}(t+\tau,\tilde X^{\tilde t , \tilde x, t,x}_{t+\tau},Du(t+\tau, \tilde X^{\tilde t ,\tilde x, t,x}_{t+\tau}),D^2u(t+\tau,\tilde X^{\tilde t , \tilde x, t,x}_{t+\tau}))}{f(\tau)}\mathcal{M}^{\tilde t,\tilde x}_{t,t+\tau}\,\mathbf{1}_{\tau<T-t}]\nonumber \\
D^2u(t,x)&=&\frac{\E[g(\tilde X_T^{ t, x, t, x})\mathcal{V}^{ t, x}_{t,T}\mathbf{1}_{\tau\geq T-t}]}{1-F(T-t)} \nonumber \\
&&+\E[ \frac{H^{ t, x}(t+\tau,\tilde X^{ t ,  x, t,x}_{t+\tau},Du(t+\tau,\tilde X^{ t , x, t,x}_{t+\tau}),D^2u(t+\tau,\tilde X^{ t ,  x, t,x}_{t+\tau}))}{f(\tau)}\mathcal{V}^{ t, x}_{t,t+\tau}\,\mathbf{1}_{\tau<T-t}]\nonumber \\
\end{eqnarray}
where $F$ is the cumulative distribution of $f$. 
We will now apply recursively this representation~\eqref{eq:RepresentTau} by considering a sequence of i.i.d. random times $(\tau_k)$. 

Let us introduce a non regular (stochastic) mesh of the interval $[t_0,T]$, 
\begin{equation}
\label{eq:pi}
\pi:=(T_0:=t_0<T_1<\cdots <T_k\cdots<T_{N_T}<T_{N_T+1}=T)\ ,
\end{equation}
characterized by the Markov chain $(T_k)$ defined by  
\begin{equation}
\label{eq:tau}
\left\{
\begin{array}{lll}
T_0&=&t_0\\
T_{k+1}&=&T_k+\delta T_{k+1}\ ,\textrm{for $k\in \mathbb{N}$ where}\\
\delta T_{k+1}&=&\tau_{k+1}\wedge (T-(T_k+\tau_{k+1}))^+\ ,
%N_T&=&\inf \{n\ \vert\ T_{n+1}\geq T\}\quad \textrm{is supposed to be finite a.s.}\ ,
\end{array}
\right .
\end{equation}
where $(\tau_k)$ is an i.i.d. sequence of random times distributed according the common probability density $f$. 
Notice that $(T_k)$ defines a Markov chain with an absorbing state, $T$. $(T_k)$ will define the so-called \textit{switching time}. \\
The random integer $N_T$ is defined as the following stopping time 
\begin{equation}
\label{eq:NT}
N_T=\inf \{n\ \vert\ T_{n+1}\geq T\}\ .
\end{equation}
Now notice by the law of large numbers that $ \frac{1}{n} \sum_{k=1}^n \tau_k \longrightarrow \E[\tau] > 0$ so $ \sum_{k=1}^n \tau_k \longrightarrow \infty$ almost surely so  $N_T$ is almost surely finite.
In the sequel, we will consider an i.i.d. sequence $(\tau_k)$ of gamma variables with parameters $(\kappa>0,\theta>0)$ recalling that the gamma density with parameter $(\kappa>0,\theta>0)$ is given by 
\begin{equation}
\label{eq:gammaDens}
f_\Gamma^{\kappa,\theta}(s)=\frac{s^{\kappa-1}e^{-s/\theta}}{\Gamma(\kappa)\theta^\kappa}\ ,\quad \textrm{for all}\ s>0\ ,
\end{equation}
where $\Gamma$ is the gamma Euler function. \\

For a given mesh $\pi$ (defined as in~\eqref{eq:pi}~\eqref{eq:tau}), we consider the following sequence (defining a Markov chain conditionally to the mesh $\pi$)
\begin{equation}
\label{eq:Xk}
\left \{
\begin{array}{l}
\bar X_0=X^{t_0,x_0}_{T_0}=x_0\\
\bar X_{k+1}=\bar X_k+b(T_k,\bar X_k)\delta T_{k+1}+\sigma(T_k,\bar X_k)\delta W_{k+1}\ ,
\end{array}
\right .
\end{equation}
where $\delta W_{k+1}:=W_{T_{k+1}}-W_{T_k}$. 
For the sake of simplicity, we will often note $b_k$ or $\sigma_k$ instead of $b(T_k,\bar X_{k})$ or $\sigma(T_k,\bar X_{k})$. 

Using representation~\eqref{eq:RepresentTau} with $(\tilde t, \tilde x) =( T_{k},\bar X_{k})$ and $\tau = \tau_{k+1}$, conditioning with respect to $(\tilde t, \tilde x)$  one gets for any integer $k\geq 0$
\begin{eqnarray}
\label{eq:RepresentTk}
u(T_k,\bar X_k)&=&\frac{\E[g(\tilde X^{T_k,\bar X_k}_T)\mathbf{1}_{T_{k+1}=T}]}{1-F(T-T_k)}+\E[ H_{k+1}\,\mathbf{1}_{T_{k+1}<T}]
\end{eqnarray}
with $ \tilde X^{T_k,X_k}_s := \tilde X^{T_k,X_k,T_k,X_k }_s$ for $s \ge T_k$ and 
$$
H_{k+1}:=\frac{H^{T_k,\bar X_k}(T_{k+1},\bar X_{{k+1}},Du(T_{k+1},\bar X_{{k+1}}),D^2u(T_{k+1},\bar X_{{k+1}}))}{f(\delta T_{k+1})}\ .
$$
The derivatives $Du$ and $D^2u$ in $H_{k+1}$ are given by applying  the representation~\eqref{eq:RepresentTau} with $(\tilde t, \tilde x) =( T_{k+1},\bar X_{k+1})$ and $\tau = \tau_{k+2}$, conditioning with respect to $(\tilde t, \tilde x)$  one gets for any integer $k\geq 0$
\begin{eqnarray*}
Du(T_{k+1},\bar X_{k+1})&=&\frac{\E[g(\tilde X^{T_{k+1},\bar X_{k+1}}_T)\mathcal{M}^{T_{k+1},\bar X_{k+1}}_{T_{k+1},T}\mathbf{1}_{T_{k+2}=T}]}{1-F(T-T_{k+1})}+\E[ H_{k+2}\mathcal{M}^{T_{k+1},\bar X_{k+1}}_{T_{k+1},T_{k+2}}\,\mathbf{1}_{T_{k+2}<T}]\nonumber \\
Du(T_{k+1},\bar X_{k+1})&=&\frac{\E[g(\tilde X^{T_{k+1},\bar X_{k+1}}_T)\mathcal{V}^{T_{k+1},\bar X_{k+1}}_{T_{k+1},T}\mathbf{1}_{T_{k+2}=T}]}{1-F(T-T_{k+1})}+\E[ H_{k+2}\mathcal{V}^{T_{k+1},\bar X_{k+1}}_{T_{k+1},T_{k+2}}\,\mathbf{1}_{T_{k+2}<T}]\ .
\end{eqnarray*}
Let us  introduce the sequence of weights $(P_{k})_{k\geq 1}$ such that for $k=1,\cdots , N_T$ 
\begin{equation}
\label{eq:weights}
\left\{
\begin{array}{lll}
P_{k+1}&=& \frac{M_{k+1}+\frac{1}{2}V_{k+1}}{f(\delta T_k)}  \ ,\\
M_{k+1}&=&  \delta b_k . (\sigma_k^{-1})^\top \frac{\delta W_{k+1}}{\delta T_{k+1}},\quad \textrm{with}\quad \delta b_k:=b_k-b_{k-1}\\
V_{k+1}&=&  \delta a_k : (\sigma_k^{-1})^\top \frac{\delta W_{k+1} \delta W_{k+1}^\top -\delta T_{k+1} \mathbb{I}}{(\delta T_{k+1})^2} \sigma_k^{-1}  \ ,\quad \textrm{with}\quad \delta a_k:=a_k-a_{k-1}\ .
\end{array}
\right .
\end{equation}
Following the same lines as the proof of Theorem~2.2 in~\cite{HTT}, one can derive by recurrence  a representation formula for $u$ as the expectation of an exactly simulatable variable.
Before one has to introduce some new assumptions. 
\begin{ass}
\label{ass:Bound}
The coefficients $b$ and $a$ are uniformly bounded % and $\sigma$ is bounded away from zero
i.e. there exists a finite constant $M$ such that for any $(t,x)\in [0,T]\times \R^d$ 
$$
%\vert g(x)\vert \leq M\ ,\quad
\|b_t(x)\|  \leq M\ ,\quad  \|a_t(x)\| \leq M\  %,\quad \textrm{and}\quad  \|\sigma_t(x)^{-1}\|  \leq M\
.
$$
\end{ass}
\begin{ass}
\label{ass:gFunc}
The function $g$ is Lipschitz.
\end{ass}
\begin{prop}
\label{prop:Representation}
Under assumptions~\ref{ass:unique2},~\ref{ass:Bound} and~\ref{ass:gFunc}
, the following representation holds
\begin{equation}
\label{eq:Representation}
\begin{array}{lll}
u(t_0,x_0)&:=&\E[g(X^{t_0,x_0}_T)]
%= \E[ \phi] := 
=\E[\frac{g(\bar X_{N_T+1})}{1-F(\delta T_{N_T+1})}\prod_{k=2}^{N_T+1}P_k]\ ,\\
\end{array}
\end{equation}
with the convention $\prod_{k\in \varnothing}=1$.
\end{prop}
\begin{rem}
\label{rem:Representation}
\begin{enumerate}
	\item Proposition~\ref{prop:Representation} proves that any $v$ satisfying the equation~\eqref{eq:Represent} is given by the above explicit equation~\eqref{eq:Representation}: this \textit{a posteriori} proves the uniqueness of the solution of~\eqref{eq:Represent}. 
	\item Using an exponential distribution for $f$, one recovers the representation given in  \cite{HTT}.
	\item Representation~\eqref{eq:Representation} is very similar to the forward representation developed in~\cite{Bally} and used in~\cite{Andersson} for the same purpose of unbiased simulation of SDEs expectations. The main interest of~\eqref{eq:Representation} being that the coefficients $b$ and $a$ have no need to be differentiable. 
	\end{enumerate}
%\item Notice that the point $(t_0,x_0)$ impacts implicitly both the process $(X)$ as an initial condition $(t_0,x_0)$ and the stochastic sequence of times $\pi :=(T_k)_{k=0,\cdots N_T+1}$ which begins at time $T_0=t_0$. 
\end{rem}
\begin{proof}
%\begin{rem}
We will only give the sketch of the proof since it mimics step by step the proof of Theorem~2.2 in~\cite{HTT}, which proceeds into two steps. 
\begin{enumerate}
	\item First suppose that Assumption~\ref{ass:unique1} is satisfied, then the representation~\eqref{eq:Representation} holds. Indeed, Lemma~\ref{lem:Represent} applies and  following the recurrence arguments developed in~\cite{HTT} implies the representation~\eqref{eq:Representation}. 
	\item  For the clarity of the paper we recall here the arguments developed in~\cite{HTT} to extend, by smooth approximations, the representation proved at item 1. outside of Assumption~\ref{ass:unique1}. 
	Since according to assumptions~\ref{ass:unique2} and~\ref{ass:gFunc}, $(b,\sigma,g)$ are Lipschtiz  we can find a sequence of bounded smooth functions $(b^\epsilon, \sigma^\epsilon, g^\epsilon)$ converging
locally uniformly to $(b, \sigma, g)$ as $\epsilon \longrightarrow 0$ such that Assumption~\ref{ass:unique1} is verified when replacing $\mathcal{L}$ by $\mathcal{L}^\epsilon$ (the infinitesimal generator associated to $(b^\epsilon,\sigma^\epsilon)$) in the PDE~\eqref{eq:PDE1}. 
Let $X^\epsilon$ denote the solution of
$$
dX^{\epsilon}_t = b^\epsilon(t,X^{\epsilon}_t) dt + \sigma^\epsilon(t,X^{\epsilon}_t) d W_t
$$
and set 
$
u^\epsilon(t_0,x_0):=\E[g^\epsilon(X^\epsilon_T)]$. By item 1. The following representation holds 
$$
 u^\epsilon(t_0,x_0)= \psi^\epsilon:= \E[\frac{g^\epsilon(\bar X_{N_T+1}^\epsilon )}{1-F(\delta T_{N_T+1})}\prod_{k=2}^{N_T+1}P^{\epsilon}_k]\ ,
$$
 where $(\bar X^\epsilon_k)$ (and respectively the weights $(P^{\epsilon}_k)$) are given by~\eqref{eq:Xk} (resp. the recursion~\eqref{eq:weights}), where $X$ is replaced by $X^\epsilon$. 
By stability of SDEs, and dominated convergence theorem, 
$
u^\epsilon(t_0,x_0) \xrightarrow[\epsilon \rightarrow\, 0]{}u(t_0,x_0)\ .
$
Similarly one can prove that 
$
\psi^\epsilon \xrightarrow[\epsilon \rightarrow\, 0]{} \E[\frac{g(\bar X_{N_T+1} )}{1-F(\delta T_{N_T+1})}\prod_{k=2}^{N_T+1}P_k]\ ,
$
 which ends the proof. 
\end{enumerate}
%\end{rem}
\end{proof}
We next define a second representation that will be interesting in order to get some finite variance estimator for some given switching distribution, $f$. Following~\cite{HTT}, one can introduce antithetic variables to control the variance induced by the last time step. 
Let $ \mathcal{G}_{T}:=\sigma(N_T, T_k, \Delta W_k\mathbf{1}_{k \le N_T}, k \ge 1)$. 
%and recall that $\pi$ denotes the stochastic sequence $(T_k)_{k=0,\cdots ,N_T+1}$ independent of the Brownian. 
Observe that 
$$
\E[M_{N_T+1}\,\vert \mathcal{G}_{T}]=\E[V_{N_T+1}\,\vert \mathcal{G}_T]=\E[P_{N_T+1}\,\vert \mathcal{G}_T]=0\ .
$$
Hence replacing $g(\bar X_{N_T+1})$ by $g(\bar X_{N_T+1})-g(\bar X_{N_T})$ in~\eqref{eq:Representation} does not change the expectation since due to the tower property:
$$
\E[\frac{ g(\bar X_{N_T})}{1-F(\delta T_{N_T+1})} \prod_{k=2}^{N_T+1}P_k]=\E[\frac{ g(\bar X_{N_T})}{1-F(\delta T_{N_T+1})} P_{N_T+1}\prod_{k=2}^{N_T}P_k]=0\ .
$$
Notice that the following decomposition holds whenever $N_T\geq 1$
$$
g(\bar X_{N_T+1})\prod_{k=2}^{N_T+1}P_k=
g(\bar X_{N_T+1})\frac{M_{N_T+1}} {f(\delta T_{N_T})} \prod_{k=2}^{N_T}P_k+\frac{1}{2}g(\bar X_{N_T+1})\frac{V_{N_T+1}} {f(\delta T_{N_T})}\prod_{k=2}^{N_T}P_k\ .
$$
Then using antithetic variables for the second term in the r.h.s. of the above equality yields the following estimator. 
\begin{prop}
\label{prop:Anti}
Under assumptions~\ref{ass:unique2},~\ref{ass:Bound} and~\ref{ass:gFunc}
, the following representation holds 
\begin{equation}
\label{eq:Anti}
u(t_0,x_0)=
\E[ \beta \prod_{k=2}^{N_T}P_k\mathbf{1}_{N_T\geq 1}]+\E[\frac{g(\bar X_1)}{1-F(\delta T_1)} \mathbf{1}_{N_T=0}]\ ,
\end{equation}
where $\beta:= \frac{1}{2}(\beta_1+\beta_2)$ with 
\begin{equation}
\label{eq:beta}
\left \{
\begin{array}{lll}
\beta_1&:=&  \frac{g(\bar X_{N_T+1})-g(\bar X_{N_T})}{1-F(\delta T_{N_T+1})} \frac{M_{N_T+1} + \frac{1}{2} V_{N_T+1}}{f(\delta T_{N_T})}, \\
\beta_2&:=&  \frac{g(\hat X_{N_T+1})-g(\bar X_{N_T})}{1-F(\delta T_{N_T+1})} \frac{-M_{N_T+1} + \frac{1}{2} V_{N_T+1}}{f(\delta T_{N_T})}
\end{array}
\right .
\end{equation}
and for any $n\in \mathbb{N}$, $\hat X_{n+1} =\bar X_{n}+b_{n}\delta T_{n+1}-\sigma_{n}\delta W_{n+1}$.
\end{prop}

\section{Variance Analysis in the case of Gamma distribution}
%-----------------------------------------------------
The previous representation given by Proposition \ref{prop:Anti}  is general but the variance associated to the estimator is generally infinite as it is the case when $f$ is an exponential density.
From now on, we will suppose that the density $f=f_\Gamma^{\kappa,\theta}$ is the Gamma density~\eqref{eq:gammaDens} with parameters $(\kappa,\theta)$  with cumulative distribution $F= F_{\Gamma}^{\kappa,\theta}$.\\
First, we will introduce the following assumptions.
\begin{ass}
\label{ass:LH}
The following assertions hold
\begin{enumerate}
%	\item The function $g$ is Lipschitz  i.e. there exists a finite constant $L$ such that for any $(x,x')\in  \R^d\times \R^d$ 
%$$
%\vert g(x)-g(x')\vert  \leq L\| x-x'\|\ .
%$$
	%\item $b$ and $a$ are uniformly $1/2$-H\"older continuous w.r.t. variable $t$ i.e. there exists a finite constant $H$ such that for any $(t,t',x)\in [0,T]\times [0,T]\times \R^d$
%$$
%\| b(t,x)-b(t',x)\| +\| \sigma(t,x)-\sigma(t',x)\| \leq H\vert t-t'\vert^{1/2}\ .
%$$ 
	\item $g$ is Lipschitz and $g\in C^{2}$.
	\item $\kappa ·\le \alpha \wedge\frac{1}{2}$. 
\end{enumerate}
\end{ass}
Now, we can state the following proposition. 
\begin{prop}
\label{prop:finVar}
Under Assumption~\ref{ass:unique2}, ~\ref{ass:Bound}
and~\ref{ass:LH}, the estimator defined by~\eqref{eq:Anti} in Proposition \ref{prop:Anti} has finite variance.  
%\begin{equation}
%\label{eq:FiniteVar}
%\E[\Big ( \beta \prod_{k=2}^{N_T}P_k\Big )^2]
%<\infty\ .
%\end{equation}
\end{prop}
\begin{proof}
Let $\bar{\mathcal{F}}_k$ denote the sigma-field generated by the Brownian  up to the random time $T_k$ and the random times up to the random time $T_{k+1}$ i.e. $\bar{\mathcal{F}}_k:=\sigma( T_1,..., T_{k+1}, (W_s)_{s \le T \wedge T_k})$.
Let us consider the second term on the r.h.s of~\eqref{eq:Anti}.  Notice that $\E[(g(\bar X_1))^2]$ can  easily be bounded by the boundness assumptions on $b$ and $\sigma$ and the Lipschitz property of $g$.  \\
Let us consider the first term on the r.h.s. of~\eqref{eq:Anti}. 
\begin{eqnarray}
\label{eq:V1}
\E[\Big ( \beta \mathbf{1}_{N_T\geq 1}\prod_{k=2}^{N_T}P_k\Big )^2]
&=&
\sum_{n=1}^\infty \E[\Big ( \beta \prod_{k=2}^{n}P_k\Big )^2\,\vert N_T=n]\mathbb{P}(N_T=n)\nonumber \\
\end{eqnarray} 
The proof will be decomposed into several steps. We will first try to bound the general term of the above series $\E[\Big ( \beta \prod_{k=2}^{n}P_k\Big )^2\,\vert N_T=n]$, then we will consider the sum. 
\begin{enumerate}
\item Bounding $\E[ \beta^2\vert \bar{\mathcal{F}}_n, N_T=n]$\\
 First considering $M_{k+1}$ and $V_{k+1}$ one easily obtains
\begin{eqnarray}
\label{eq:MVBound}
\E[M^4_{k+1}\,\vert \bar{\mathcal{F}}_k, N_T=n]
& \leq &  C\frac{(\delta b_k)^4}{(\delta T_{k+1})^2}
\ , \nonumber  \\
\quad \E[V^4_{k+1}\,\vert \bar{\mathcal{F}}_k,  N_T=n] &\leq& C\frac{(\delta a_k)^4}{(\delta T_{k+1})^4}\ .
\end{eqnarray}
Notice that in the sequel, $C$ will denote finite constants that may change from line to line that do not depend on $k$ or $n$ but only on the characteristics of the problem ($T$, the bounds or Lipschitz constants related to $g$, $b$, $\sigma$, $a$). 
Then  consider the general term of the sum \eqref{eq:V1}.
\begin{eqnarray*}
\E[ \beta^2\prod_{k=2}^{n}P^2_k\,\vert N_T=n]
&=&
\E \Big [\E[\beta^2\vert   \bar{\mathcal{F}}_n,  N_T=n] \Big( \prod_{k=2}^{n}P_k\Big )^2\,\vert N_T=n\Big ]\ .
\end{eqnarray*}
We get
\begin{align}
  \label{eq:decompBeta}
  \E[ \beta^2 &\vert  \bar{\mathcal{F}}_n,  N_T=n]   \nonumber \\
 &  \le \frac{C}{(1-F_{\Gamma}^{\kappa,\theta}(T))^2} \E[ (g(\bar X_{n+1})-g(\hat X_{n+1}))^2  \frac{M_{n+1} ^2}{f_{\Gamma}^{\kappa,\theta}(\delta T_n)^2} \vert  \bar{\mathcal{F}}_n, N_T=n] +  \\
&    \frac{C}{(1-F_{\Gamma}^{\kappa,\theta}(T))^2} \E[ (g(\bar X_{n+1})+g(\hat X_{n+1}) -2 g(\bar X_n))^2 \frac{V_{n+1}^2}{f_{\Gamma}^{\kappa,\theta}(\delta T_n)^2} \vert  \bar{\mathcal{F}}_n, N_T=n]  \nonumber
\end{align}
Consider the first term on the r.h.s. of inequality~\eqref{eq:decompBeta},
by the Lipschitz property of $g$, the boundness of  $b$,$ \sigma$ and using the fact that $\sigma$ is uniformly bounded away from zero, we obtain
\begin{align}
  \label{eq:firstRSH}
  \E[\vert g(\bar X_{n+1})-g(\hat X_{n+1})\vert^2   & \frac{M_{n+1} ^2}{(f_{\Gamma}^{\kappa,\theta}(\delta T_n))^2}  \vert \bar{\mathcal{F}}_n, N_T=n]
    \nonumber \\
 & \leq   C \frac{ \|\delta b_n\|^2}{(f_{\Gamma}^{\kappa,\theta}(\delta T_n))^2} \leq C
  \|\delta b_n\|^2 (\delta T_n)^{2(1-\kappa)}
\end{align}
Consider the second term of \eqref{eq:decompBeta}. By Assumption~\ref{ass:LH}.2 ($g\in C^{2}$) one can apply Ito and obtain
\begin{eqnarray*}
\vert g(\bar X_{n+1})+g(\bar X_{n}+b_{n}\delta T_{n+1}-\sigma_{n}\delta W_{n+1})-2g(\bar X_{n})\vert
%&=&\vert \int_0^{\delta T_{n+1}} [Dg(X_s)+\frac{1}{2} D^2g(X_s)]\,ds\vert \\
&\leq &C \delta T_{n+1}\ .
\end{eqnarray*}
This implies still using the boundness of  $b$, $\sigma$ and using the fact that $\sigma$ is uniformly bounded away from zero :
\begin{align}
\label{eq:secondRSH}
\E[(g(\bar X_{n+1})+g(\hat X_{n+1})& -2 g(\bar X_n))^2   \frac{V_{n+1}^2}{f_{\Gamma}^{\kappa,\theta}(\delta T_n)^2} \vert  \bar{\mathcal{F}}_n, N_T=n ]
   \nonumber \\
& \leq  C(\delta T_{n+1})^2 \frac{\|\delta a_n\|^2}{\delta T_{n+1}^2 f_{\Gamma}^{\kappa,\theta}(\delta T_{n})^2 }  \le  C(\delta T_n)^{2(1-\kappa)} \|\delta a_n\|^2\ .
\end{align}
Injecting~\eqref{eq:firstRSH} and~\eqref{eq:secondRSH} into~\eqref{eq:decompBeta} finally yields 
\begin{equation}
\label{eq:BoundBeta}
\E[ \beta^2\,\vert   \bar{\mathcal{F}}_n, N_T=n]\leq C(\delta T_n)^{2(1-\kappa)}\Big (\|\delta b_n\|^2+\|\delta a_n\|^2\Big )\ .
\end{equation}
%
%%%%%%%%%%%%%%%%%%%%%%%%%%%%%%%%%%%%%%%%
\item Bounding $\E[ C_k^2\,\vert   \bar{\mathcal{F}}_{k-1}, N_T=n]$, where the r. v. $C_k$ is defined by 
\begin{equation}
\label{eq:Ck}
%C_{k}:=\Big (\|\delta b_{k}\|^2+\frac{\|\delta a_{k}\|^2}{\delta T_{k+1}}\Big ) (\delta T_{k})^{2(1-\kappa)}\ .
C_{k}:=\|\delta b_{k}\|^2+\|\delta a_{k}\|^2\ .
\end{equation}
Consider the term $\Vert \delta b_k\Vert$, 
\begin{align*}
\E[ \|\delta b_k\|^4\,\vert   \bar{\mathcal{F}}_{k-1}, & N_T=n] \\
& = \E[\| b(T_k,\bar X_{k})-b(T_{k-1},\bar X_{{k-1}})\|^4\,\vert   \bar{\mathcal{F}}_{k-1}, N_T=n]  \\
 &  \leq \nonumber 8\E\big [\| b(T_k,\bar X_{k})-b(T_k,\bar X_{{k-1}})\|^4+   \\
&\quad \quad \quad  \| b(T_k,\bar X_{{k-1}})-b(T_{k-1},\bar X_{{k-1}})\|^4\,\vert \bar{\mathcal{F}}_{k-1},  N_T=n\big ]  \nonumber \\
&\leq  C ( 1 + (\delta T_k)^2)  (\delta T_k)^2 +  C (\delta T_k)^{4\alpha} \le  C (\delta T_k)^{4(\alpha\wedge \frac{1}{2})}
\end{align*}
using the fact that $b$ is Lipschitz  w.r.t. the space variable and $\alpha$-H\"older continuous w.r.t. the time variable. With the same development on $\delta a_k$ one finally gets 
\begin{equation}
\label{eq:ab}
\E[ C_k^2\,\vert  \bar{\mathcal{F}}_{k-1},  N_T=n]
\leq  C(\delta T_k)^{4(\alpha \wedge \frac{1}{2})}\ ,
\end{equation}
%
%%%%%%%%%%%%%%%%%%%%%%%%%%%%%%%
\item Bounding $\E[P^4_{k+1}\,\vert   \bar{\mathcal{F}}_k,  N_T=n]$\\
Using~\eqref{eq:MVBound}, we obtain 
\begin{align}
\label{eq:BoundPk}
\E[P^4_{k+1}\,\vert    \bar{\mathcal{F}}_k, &  N_T=n] \nonumber \\
&= 
\E[(M_{k+1}+\frac{1}{2}V_{k+1})^4 (\delta T_k)^{4(1-\kappa)} \theta^{4\kappa } \Gamma^4(\kappa) e^{4 \delta T_k /\theta} \Big \vert \bar{\mathcal{F}}_k, N_T=n]\nonumber \\
&\leq 
C\Big (\|\delta b_k\|^4+\frac{\|\delta a_k\|^4}{(\delta T_{k+1})^2}\Big )\frac{(\delta T_k)^{4(1-\kappa)}}{(\delta T_{k+1})^2} \nonumber\\
&\leq  
%C\frac{C_k^2}{(\delta T_{k+1})^2}\ ,
C C_k^2\frac{(\delta T_k)^{4(1-\kappa})}{(\delta T_{k+1})^4}\ ,
\end{align}
observing that $\delta T_{k+1}\leq T\leq C$ and recalling that $C_k$ is defined by~\eqref{eq:Ck}.  
Using the tower property of expectation and bound~\eqref{eq:BoundBeta} yields
\begin{align*}
\E[ \beta^2\prod_{k=1}^{n-1} & P^2_{k+1}\,\vert N_T=n] \\
&=
\E\Big [E[ \beta^2\,\vert  \bar{\mathcal{F}}_n, N_T=n]\prod_{k=1}^{n-1} P^2_{k+1}\,\vert N_T=n\Big ]\\
&\leq 
C\E\Big [(\delta T_n)^{2(1-\kappa)}\Big (\|\delta a_n\|^2+\|\delta b_n\|^2\Big )\prod_{k=1}^{n-1} P^2_{k+1}\,\vert N_T=n\Big ]\\
&\leq  
C\E\Big [\E[(\delta T_n)^{2(1-\kappa)}C_nP^2_n\,\vert  \bar{\mathcal{F}}_{n-1},  N_T=n]\prod_{k=1}^{n-2} P^2_{k+1}\,\vert N_T=n\Big ]\ .
\end{align*}

By Cauchy-Schwarz, for any $k$ and using~\eqref{eq:ab}, we have 
\begin{align}
\label{eq:cPk}
\E[C_{k}P^2_{k}\vert  \bar{\mathcal{F}}_{k-1},&  N_T=n]  \nonumber \\
&\leq 
\Big (\E[C^2_{k}\vert   \bar{\mathcal{F}}_{k-1},  N_T=n]\Big )^{1/2}\Big (\E[P^4_{k}\vert  \bar{\mathcal{F}}_{k-1}, N_T=n]\Big )^{1/2} \nonumber \\
&\leq 
C (\delta T_k)^{2(\alpha\wedge \frac{1}{2})} C_{k-1}\frac{(\delta T_{k-1})^{2(1-\kappa)}}{(\delta T_k)^2} \nonumber \\
&\leq 
C C_{k-1}\frac{(\delta T_{k-1})^{2(1-\kappa)}}{(\delta T_k)^{2((1-\alpha)\vee \frac{1}{2})}}\ .
\end{align}
Hence, we obtain by recursion
\begin{equation}
\label{eq:Revise}
\E[ \beta^2\prod_{k=1}^{n-1} P^2_{k+1}\,\vert N_T=n]
\leq  
C^{n+1}\E[(\delta T_n)^{2(1-\kappa)}\prod_{k=1}^{n-1}\frac{(\delta T_k)^{2(1-\kappa)}}{(\delta T_{k+1})^{2((1-\alpha)\vee \frac{1}{2})}}\,\vert N_T=n]\ ,
\end{equation}
observing that $\E[C_1\vert \bar{\mathcal{F}}_0,  N_T=n]\leq C(\delta T_1)^{2(\alpha\wedge \frac{1}{2})}\leq C T$.

Then recalling that $\kappa\leq \alpha\wedge \frac{1}{2}$ implies $(\delta T_k)^{2((\frac{1}{2}-\kappa)\wedge (\alpha-\kappa))}\leq T^{2((\frac{1}{2}-\kappa)\wedge (\alpha-\kappa))}$ finally yields  
\begin{equation}
\label{eq:prodBpound}
\E[ \beta^2\prod_{k=1}^{n-1} P^2_{k+1}\,\vert N_T=n]
\leq
CC^{n}\ .
\end{equation}
%
%
%%%%%%%%%%%%%%%%%%%%%%%%%%%%%%%%%%%%%%%%%%%%%%%%%%%%%%%%%%%
\item Convergence of the sum $\sum_{n=1}^\infty C^n\mathbb{P}(N_T=n)$.
Let us introduce
$
S_n=\sum_{k=1}^n\tau_k\ ,
$
notice that 
$S_n\sim \Gamma(n\kappa , \theta)$ with cumulative distribution 
$$
F_{S_{n}}(s)=\int_0^s \frac{r^{n\kappa-1}e^{-r/\theta}}{\Gamma(n\kappa)\theta^{n\kappa}} dr\ .
$$
Hence one can bound $\mathbb{P}(N_T=n)$ as follows 
\begin{eqnarray*}
\mathbb{P}(N_T=n)&\leq &\mathbb{P}(N_T=n)\\
%&\leq & \mathbb{P}(S_n\leq T)-\mathbb{P}(S_{n+p}\leq T)\\
& \leq & \mathbb{P}(S_n\leq T) \\
&\le & \int_0^T \frac{r^{n\kappa-1}}{\Gamma(n\kappa)\theta^{n\kappa}} dr =  \frac{T^{n\kappa}}{ n\kappa \Gamma(n\kappa)\theta^{n\kappa}}
\end{eqnarray*}
This implies that 
$$
\sum_{n=1}^\infty C^n\mathbb{P}(N_T=n)
\leq 
 \sum_{n=1}^\infty \frac{\hat C ^n}{n \kappa \Gamma (n \kappa)}\ ,
 $$
 with $\hat C = C \frac{T}{\theta}$.
Using the generalization of the Stirling formula  $\Gamma (z)\sim z^{z-1/2} e^{-z} \sqrt{2\pi}$ one proves that
$\frac{\hat C^n}{ n \kappa \Gamma (n \kappa)} \sim  \frac{ \hat{C}^{\frac{1}{2 \kappa}}  e^{-\frac{1}{2}}} {\sqrt{2 \pi}} ( \frac{\hat{C}^{\frac{1}{\kappa}} e}{n\kappa} )^{n \kappa +\frac{1}{2}}$ which 
is the general term of a convergent sum. 
\end{enumerate}
\end{proof}
\begin{rem}
  The convergence of the series~\eqref{eq:V1} relies on  two facts :
  \begin{itemize}
  \item The general term of the series~\eqref{eq:V1} has to be finite:
	$
	\E[\Big ( \beta \prod_{k=2}^{n}P_k\Big )^2\,\vert N_T=n]<\infty \ ,
	$
	for any fixed number of switching times $N_T=n$. However, one can observe  that our bound on the r.h.s. of~\eqref{eq:Revise} can possibly blow up to infinity when $\kappa> \alpha\wedge 1/2$. In particular, in the case of an exponential density, corresponding to $\kappa=1$, it is well-known that the conditional distribution, $\mathcal{L}(\delta T_k\vert N_T=n)$,  is the uniform distribution on $[0,T]$, hence the expectation on the r.h.s. of~\eqref{eq:Revise} is infinite. 
	%In the exponential case one can prove that the general term of the series~\eqref{eq:V1} is indeed infinite.  
	%Conditionally to the number of jumps, $N_T=n$, one has to be able to bound for all $k< n$ $\E[ \frac{1}{(f(\delta T_k))^2  \delta T_k} \vert N_T=n ]$, where $f$ is the underlying probability density of $\delta T_k$. In the case of an exponential density, it is well-known that the conditional distribution, $\mathcal{L}(\delta T_k\vert N_T=n)$,  is the uniform distribution on $[0,T]$, hence this integral is infinite. In the case of a Gamma distribution, $ \frac{1}{(f(\delta T_k))^2  \delta T_k} \leq C \delta T_k^{1-2\kappa}\leq C (\sup (T,1))^{1-2\kappa}$ with $\kappa \leq \frac{1}{2}$  (the occurrence of small jumps is high) sothe integral 
	When $\kappa\leq \alpha\wedge 1/2$,   we observe that our bound is finite whatever the conditional  distribution, $\mathcal{L}(\delta T_k\vert N_T=n)$. Notice that using gamma switching times  increases the occurrence of small jumps w.r.t. the exponential case and hence the occurence of high numbers of time steps is also increased. To better adjust the complexity and  variance tradeoff, one could consider other switching times densities with a smaller intensity of small jumps and rely on the conditional law $\mathcal{L}(\delta T_k\vert N_T=n)$ to ensure that the expectation on the r.h.s. of~\eqref{eq:Revise} is bounded.
    \item The sum $\sum_{k=1}^\infty C^n\mathbb{P}(N_T=n)$ has to converge. By increasing the intensity of small jumps as explained at point 1., we expect that $\mathbb{P}(N_T=n)$ will decrease more slowly with $n$. This results in a tradeoff one has to achieve: increasing small jumps intensity to be able to bound each term of the series but not too strongly to ensure the convergence of $\sum_{k=1}^\infty C^n\mathbb{P}(N_T=n)$. 
    \end{itemize}
\end{rem}
Consequently, the representation~\eqref{eq:Anti} provides a Monte Carlo approach to compute $\E[g(X_T)]$, by simulating the regime switching process~\eqref{eq:Xk} instead of the SDE~\eqref{eq:sde} which would potentially require to implement a stochastic Euler discretization scheme. However, even though our estimator is proved to have finite variance, one can observe in practice huge variances due to the product of a random number of terms $P_k$ that could potentially take values greater that one. This expectation of products is \textit{by nature} not a good candidate for Monte Carlo estimation.  Hence, we propose to use a resampling procedure to change this expectation of products in a product of expectations which is known to be much more stable for estimation.

\section{Resampling method for regime switching processes}
%-------------------------------------------
In this section, we propose to introduce  an interacting particle system (in the same vein as those thoroughly discussed in the reference books~\cite{DelMoral} and~\cite{DelMoral2}) to approximate $u(t_0,x_0)$. We will prove that the resulting estimator has finite variance under the same assumptions required to bound the variance of estimator \eqref{eq:Anti}. However, in practice, the new estimator relying on interacting particle systems will show better performances providing smaller variances in many examples, as illustrated in Section~\ref{sec:simu}. 
%The following approach will be used to stabilize the estimator \eqref{eq:Anti} is bounded, the following approach can be used. In the sequel we detail it in the case of gamma distributions.

\subsection{A Feynman-Kac measure representation}
%-------------------------------------------
First we have to express $u(t_0,x_0)$ as an integral according to a Feynman-Kac measure. 
Let us consider the Markov chain consisting of the sequence of random variables $\check X_k:=(T_k,\bar X_k)$, where $(T_k)$ and $(\bar X_k)$ are given respectively by the dynamics~\eqref{eq:tau} and~\eqref{eq:Xk}. In the sequel, we note  $\check X_{0:k}:=(\check X_0,\cdots ,\check X_{k})$ the path valued Markov chain.
Let us introduce, for any integer $k\geq 0$, the real valued function $\check G_{k}$ depending on the path   $\check x_{0:k}\in E_{k}:=(\R_+\times \R^d)^{k+1}$ with the notations $\check x_{0:k}:=(\check x_0,\cdots ,\check x_{k})$ and  $\check x_{p}:=(t_p,x_p)\in \R_+\times  \R^d$ such that 
\begin{equation}
\label{eq:tildeG}
\check G_{k}(\check x_{0:k})
:=
\left \{
\begin{array}{ll}
1 & \textrm{if}\ k=0\ \textrm{or}\ k=1 \\
\frac{\check M_{k}(\check x_{0:k})+
\frac{1}{2}\check V_{k}(\check x_{0:k})}{f_{\Gamma}^{\kappa,\theta}(\delta t_{k-1})}
& \textrm{if}\ k\geq 2 \ \textrm{and}\ \delta t_{k-1}\delta t_{k}>0 \\
%1 &\textrm{if}\ k\geq 2\ \textrm{and}\ \delta t_{k}=0\\
1 &\textrm{elsewhere}
\ .\end{array}
\right .
\end{equation}
with $\delta t_{k+1}:=t_{k+1}-t_k$ and where the real valued functions $\check M_{k+1}\,,\ \check V_{k+1}$ and $\check{\delta W}_{k+1}$ are such that for any $\check x_{0:k+1}\in  E_{k+1}$
\begin{eqnarray}
\label{eq:tildeMV}
\check M_{k+1}(\check x_{0:k+1})&:=& 
\left\{
\begin{array}{l}
(b(t_k,x_k)-b(t_{k-1},x_{k-1}) ). (\sigma(t_k,x_k)^{-1})^\top \frac{\check{\delta W}_{k+1}(\check x_{0:k+1})}{\delta t_{k+1}}\quad \textrm{if}\ \delta t_{k+1}>0\\
1\quad \textrm{elsewhere}
\end{array}
\right .
\nonumber \\
\check V_{k+1}(\check x_{0:k+1})&:=&  
\left \{
\begin{array}{l}
(a(t_k,x_k)-a(t_{k-1},x_{k-1})) : \frac{B_{k+1}(\check x_{0:k+1})}{(\delta t_{k+1})^2} \quad \textrm{if}\ \delta t_{k+1}>0\\
1\quad \textrm{elsewhere}\ ,
\end{array}
\right . 
\end{eqnarray}
with 
\begin{eqnarray*}
B_{k+1}(\check x_{0:k+1}) &:=& (\sigma(t_k,x_k)^{-1})^\top \Big (\check{\delta W}_{k+1}(\check x_{0:k+1}) \check{\delta W}_{k+1}(\check x_{0:k+1})^\top -\delta t_{k+1} \mathbb{I}\Big ) \sigma(t_k,x_k)^{-1} \nonumber 
%\tilde V_{k+1}(\tilde x_{0:k+1})&:=&  (a(t_k,x_k)-a(t_{k-1},x_{k-1})) : B_{k+1}(\tilde x_{0:k+1}) \nonumber \\
%B_{k+1}(\tilde x_{0:k+1}) &:=& (\sigma(t_k,x_k)^{-1})^T \frac{\tilde{\delta W}_{k+1}(\tilde x_{0:k+1}) \tilde{\delta W}_{k+1}(\tilde x_{0:k+1})^T -\delta t_{k+1} \mathbb{I}}{(\delta t_{k+1})^2} \sigma(t_k,x_k)^{-1} \nonumber 
\\
%\big (a_{t_k}(x_k)-a_{t_{k-1}}(x_{k-1})\big )\frac{(\tilde {\delta W}_{k+1}(\tilde x_{0:k+1}))^2-\delta t_{k+1}}{(\sigma_{t_k}(x_k))^2(\delta t_{k+1})^{3/2}}\nonumber \\
\check {\delta W}_{k+1}(\check x_{0:k+1})&:=& \sigma(t_k,x_k)^{-1} (x_{k+1}-x_k-b(t_k,x_k)\delta t_{k+1})\ .
\end{eqnarray*}
Observe that $\check G_{k+1}$ does not really depend on the whole path $\check x_{0:k+1}$, but only on $(\check x_{k-1},\check x_k,\check x_{k+1})$, for $k>0$. Recalling~\eqref{eq:weights}, notice that the following identity holds 
$$
\check G_{k}(\check X_{0:k})=P_{k}\ ,\quad \mathbb{P}\ a.s.\quad \textrm{for all}\quad k=2,\cdots ,N_T\ .
$$
In the sequel, it will appear to be crucial to consider positive \textit{potential functions} with uniformly bounded conditional variances, more specifically such that $\sup_{\check x_{0:k}\in E_k}\E[G^2_{k+1}(\check X_{0:k+1})\,\vert \check X_{0:k}=\check x_{0:k}]<\infty $, thus we define the potential functions $(G_{k})_{k\geq 0}$ (depending implicitly on $T$) such that for any $k\geq 0$ and for any $\check x_{0:k}\in E_{k}$, 
\begin{equation}
\label{eq:G}
G_{k}(\check x_{0:k}):= 
\left \{
\begin{array}{ll}
1&\textrm{if}\ k=0\\
\vert \check G_{1}(\check x_{0:1})\vert (\delta t_{1})^{1-\kappa}\sqrt{c_1(\check x_{0:1})}&
%\textrm{if}\ k=1,\  t_2<T,\ \textrm{and}\ \delta t_2>0\\
\textrm{if}\ k=1,\  \delta t_1>0\\
\vert \check G_{k}(\check x_{0:k})\vert \sqrt{\frac{c_k(\check x_{0:k})}{c_{k-1}(\check x_{0:k-1})}}\Big (\frac{\delta t_{k}}{\delta t_{k-1}}\Big )^{1-\kappa}&\textrm{if}\ k\geq 2\ ,\ \delta t_{k-1}\delta t_{k}>0\ ,\\
1 &\textrm{elsewhere}\ ,
\end{array}
\right .
\end{equation}
where the real valued function $c_k$ is defined on $E_k$, for $k\geq 1$, by 
\begin{equation}
\label{eq:ckk}
c_k(\check x_{0:k})=
\vert \delta t_k\vert +\Vert b(t_k,x_k)-b(t_{k-1},x_{k-1})\Vert^2+\Vert a(t_k,x_k)-a(t_{k-1},x_{k-1})\Vert^2
\end{equation}
%\begin{equation}
%\label{eq:ckk}
%c_k(\tilde x_{0:k})=
%\left \{
%\small{
%\begin{array}{ll}
%\vert \delta t_k\vert +\Vert b(t_k,x_k)-b(t_{k-1},x_{k-1})\Vert^2+\Vert a(t_k,x_k)-a(t_{k-1},x_{k-1})\Vert^2 & \textrm{if}\ \delta t_k\delta t_{k+1}>0\\
%1&\textrm{elsewhere}\ .
%\end{array}
%}
%\right .
%\end{equation}
Notice that this definition of $c_k$ is such that $c_k(\check X_{0:k})=C_k+\delta T_k$ where $C_k$ was defined in~\eqref{eq:Ck}, hence
$$
G^2_{k}(\check X_{0:k})=\frac{C_k+\delta T_k}{C_{k-1}+\delta T_{k-1}} \Big (\frac{\delta T_{k}}{\delta T_{k-1}}\Big )^{2(1-\kappa)}P^2_{k}\ ,\quad \mathbb{P}\ a.s.\quad \textrm{for all}\quad k=2,\cdots ,N_T\ .
$$
Then observe that one can prove an inequality similar as~\eqref{eq:cPk} with $C_k$ replaced by $c_k(\check X_{0:k})$
\begin{align}
\label{eq:cPkbis}
\E[c_{k}(\check X_{0:k})P^2_{k} & \vert   \bar{\mathcal{F}}_{k-1}, N_T=n] \nonumber \\
 & \leq 
\Big (\E[c^2_{k}(\check X_{0:k})\vert  \bar{\mathcal{F}}_{k-1},  N_T=n]\Big )^{1/2}\Big (\E[P^4_{k}\vert  \bar{\mathcal{F}}_{k-1},  N_T=n]\Big )^{1/2} \nonumber \\
&\leq 
C c_{k-1}(\check X_{0:k-1})\frac{(\delta T_{k-1})^{2(1-\kappa)}}{(\delta T_{k})^{2((1-\alpha)\vee \frac{1}{2})}}\ ,
\end{align}
which yields as announced, that for any  $\kappa \leq \alpha\wedge \frac{1}{2}$ and  $\check x_{0:k-1}\in  E_{k-1}$ 
\begin{equation}
\label{eq:BoundGk}
\E[G^2_{k}(\check X_{0:k})\,\vert \check X_{0:k-1}=\check x_{0:k-1}]\leq C<\infty \ .
\end{equation}

Notice that     
 $
\prod_{k=2}^{N_T}P_k=H_{N_T+1}(\check X_{0:N_T+1}) \prod_{k=0}^{N_T}G_k(\check X_{0:k})S_k(\check X_{0:k})\ ,\quad \mathbb{P}\ a.s.
$   
 where for any $k\geq 0$ and for any $\check x_{0:k}\in E_{k}$, 
\begin{equation}
\label{eq:S}
S_{k}(\check x_{0:k}):=Sign(\check G_{k}(\check x_{0:k}))\ ,%\quad \textrm{for any}\quad x_{0:k+1}\in (\R_+\times \R^d)^{k+2}\ ,
\end{equation}
and
\begin{equation}
\label{eq:HH}
H_{k+1}(\check x_{0:k+1}):=
\left \{
\begin{array}{ll}
%\vert \tilde G_{k+1}(\tilde x_{0:k+1})\vert 
\frac{1}{(\delta t_{k})^{1-\kappa}\sqrt{c_{k}(\check x_{0:k})}} &\textrm{if}\ k\geq 1\ \textrm{and}\ \delta t_k>0 \\
%1& \textrm{if}\ \delta t_{k}=0\ .
1& \textrm{elsewhere}\ .
\end{array}
\right . 
\end{equation}

Let us  introduce $\beta_{n+1}:= \frac{1}{2}\beta_{1,n+1} +\frac{1}{2}\beta_{2,n+1}$ defined on $E_{n+1}$ such that 
$\beta_{1,1}(\check x_{0:1})= \beta_{2,1}(x_{0:n+1}) = \frac{1}{(1- F_{\Gamma}^{\kappa,\theta}(\delta t_1))} g(x_1)$  and 
for any $n\geq 1$ 
\begin{equation}
\label{eq:betaFunc}
\left \{
\begin{array}{lll}
\beta_{1,n+1}(\check x_{0:n+1})&:=&  \frac{g(x_{n+1})-g(x_{n})} {1- F_{\Gamma}^{\kappa,\theta}(\delta t_{n+1})} \frac{ M_{n+1}(\check x_{0:n+1}) + \frac{1}{2} \check V_{n+1}(\check x_{0:n+1})} {f_{\Gamma}^{\kappa,\theta}(\delta t_{n})}\\
\beta_{2,n+1}(\check x_{0:n+1})&:=&  \frac{g(\hat x_{n+1})-g(x_{n})}{1- F_{\Gamma}^{\kappa,\theta}(\delta t_{n+1}) }  \frac{ -M_{n+1}(\check x_{0:n+1}) + \frac{1}{2} \check V_{n+1}(\check x_{0:n+1})}{f_{\Gamma}^{\kappa,\theta}(\delta t_{n})} \ ,
\end{array}
\right .
\end{equation}
with $\hat x_{n+1} = x_{n}+b(t_n,x_n) \delta t_{n+1}-\sigma(t_n,x_n) \check {\delta W}_{n+1}(\check x_{0:n+1})$.\\
Recalling~\eqref{eq:Anti}, observe that 
\begin{eqnarray}
\label{eq:uG}
u(t_0,x_0)&=&
\E[\beta \prod_{k=2}^{N_T}P_k\mathbf{1}_{N_T\geq 1}]+\E[\frac{g(\bar X_1)}{1- F_{\Gamma}^{\kappa,\theta}(\delta T_1)} \mathbf{1}_{N_T=0}]\nonumber\\
&=&
\E[ (\beta_{N_T+1}H_{N_T+1})(\check X_{0:N_T+1})(S_{0:N_T}G_{0:N_T})(\check X_{0:N_T})]
\ ,
\end{eqnarray}
%%%%%%%%%%%%%%%%%%%%%%%
    %{\bf XW
%
 %\begin{eqnarray}
%\label{eq:uG}
%u(t_0,x_0)&=&
%\E[\beta \prod_{k=2}^{N_T}P_k\mathbf{1}_{N_T\geq 1}]+\E[\frac{g(X_1)}{1- F_{\Gamma}^{\kappa,\theta}(\delta T_1)} \mathbf{1}_{N_T=0}]\nonumber\\
%&=&
%\E[ (\beta_{N_T+1}H_{N_T+1})(\tilde X_{0:N_T+1})(S_{0:N_T}G_{0:N_T})(\tilde X_{0:N_T})]
%\ ,
%\end{eqnarray}     
      %}
%%%%%%%%%%%%%%%%%%%%%%%
where to simplify the notation  $G_{p:q}$ (resp. $S_{p:q}$) denotes the product $\prod_{k=p}^q G_k$ (resp. $\prod_{k=p}^q S_k$), with in particular $G_{p,q}= \mathbf{1}$ when $p>q$, where $\mathbf{1}$ denotes the function which takes the unique value $1$.
Now, we can define the sequence of non negative measures $(\gamma_k)_{k\geq 0}$ such that for any real valued bounded test function $\varphi$ defined on $E_n:=(\R_{+}\times \R^d)^{n+1}$, we have 
\begin{equation}
\label{eq:gamma}
%\left\{
%\begin{array}{lll}
%\gamma_0(\varphi)&:=&\E[\varphi (\tilde X_{0})]\quad \textrm{i.e.}\quad \gamma_0=\mathcal{L}(\tilde X_0)\\
\gamma_k(\varphi):=\E[\varphi (\check X_{0:k})\prod_{p=0}^{k-1} G_p(\check X_{0:p})]=\E[\varphi (\check X_{0:k})G_{0:k-1}(\check X_{0:k-1})]\ ,\quad \textrm{for}\ k\geq 1\ .
%\end{array}
%\right .
\end{equation}
%In particular observe that since $G_0=1$ 
We set by convention  $\gamma_0:=\mu_0$ where $\mu_0$ denotes the probability distribution, $\mathcal{L}(\check X_0)$, of the initial condition $\check X_0=(t_0,x_0)$ i.e. $\mu_0:=\mathcal{L}(\check X_0)=\delta_{(t_0,x_0)}$.
Gathering~\eqref{eq:uG} together with the above definition one readily obtains the following proposition expressing $u(t_0,x_0)$  as an integral w.r.t. the non-negative measures $\gamma_n$.
   \begin{rem}
     The weights used  in equation \eqref{eq:G} can be generalized with $\rho \in [\frac{1}{2},1-\kappa]$ as 
     \begin{equation}
       \label{eq:Gbis}
G_{k}(\check x_{0:k}):= 
       \left \{
       \begin{array}{ll}
         1&\textrm{if}\ k=0\\
         \vert \check G_{1}(\check x_{0:1})\vert (\delta t_{1})^{\rho}\sqrt{c_1(\check x_{0:1})}&
				%\textrm{if}\ k=1,\  t_2<T,\ \textrm{and}\ \delta t_2>0\\
				\textrm{if}\ k=1,\  \delta t_1>0\\
         \vert \check G_{k}(\check x_{0:k})\vert \sqrt{\frac{c_k(\check x_{0:k})}{c_{k-1}(\check x_{0:k-1})}}\Big (\frac{\delta t_{k}}{\delta t_{k-1}}\Big )^{\rho}&\textrm{if}\ k\geq 2\ ,\ \delta t_{k-1}\delta t_{k}>0\ ,\\
         1 &\textrm{elsewhere}\ ,
       \end{array}
       \right .
     \end{equation}
   \end{rem}
\begin{prop}
\label{prop:ugamma}
Under Assumptions~\ref{ass:unique2} and~\ref{ass:Bound},
the following identity holds for any $n\geq 1$
\begin{equation}
\label{eq:ugamma}
u(t_0,x_0)=\gamma_n(\varphi_n)\ ,
\end{equation}
where $(\varphi_n)_{n\geq 1}$ is a sequence of real valued functions such that for any $n\geq 1$ and $\check x_{0:n}\in E_n:=(\R_+\times \R^d)^{n+1}$ 
\begin{equation}
\label{eq:varphi}
\varphi_n(\check x_{0:n}):=\E[(\beta_{N_T+1}H_{N_T+1})(\check X_{0:N_T+1})(S_{1:N_T} G_{n:N_T})(\check X_{0:N_T})\,\vert \check X_{0:n}=\check x_{0:n}]\ .
\end{equation}
%%%%%%%%%%%%%%%%%%%%%%%%%%%%%%%%%%%%%%%%%%%
    %{\bf XW
%\begin{equation}
%\label{eq:varphi}
%\varphi_n(\tilde x_{0:n}):=\E[(\beta_{N_T+1}H_{N_T+1})(\tilde X_{0:N_T+1})(S_{1:N_T} G_{n:N_T})(\tilde X_{0:N_T})\,\vert \tilde X_{0:n}=\tilde x_{0:n}]\ .
%\end{equation}
      %}
%%%%%%%%%%%%%%%%%%%%%%%%%%%%%%%%%%%%%%%%%%%%
\end{prop}
\begin{rem}
\label{rem:varphin}
Observe that for a given $n\geq 1$, $\varphi_n$ is defined by~\eqref{eq:varphi} as a conditional expectation of a terminal payoff delivered at a future random time $N_T+1$, knowing the state of the Markov chain from time $0$ to $n$. Hence, evaluating $\varphi_n(\check x_{0:n})$ is  not trivial, for a given $\check x_{0:n}$, since it requires to compute a conditional expectation. However whenever $\check x_n=(t_n,x_n)$ is such that $t_n\geq T$, then the knowledge of $\check X_{0:n}=\check x_{0:n}$ determines completely both $N_T=q<n$ and $\check X_{0:N_T+1}$, which implies
\begin{eqnarray*}
\varphi_n(\check x_{0:n})&:=&\E[(\beta_{N_T+1}H_{N_T+1})(\check X_{0:N_T+1})(S_{1:N_T}G_{n:N_T})(\check X_{0:N_T})\,\vert \check X_{0:n}=\check x_{0:n}]\\
&=&(\beta_{q+1}H_{q+1})(\check x_{0:q+1}) S_{1:q}(\check x_{0:q})\ .
\end{eqnarray*}
%%%%%%%%%%%%%%%%%%%%%%%%%%%%%%%%%%
%{\bf XW
%\begin{eqnarray*}
%\varphi_n(\tilde x_{0:n})&:=&\E[(\beta_{N_T+1}H_{N_T+1})(\tilde X_{0:N_T+1})(S_{1:N_T}G_{n:N_T})(\tilde X_{0:N_T})\,\vert \tilde X_{0:n}=\tilde x_{0:n}]\\
%&=&(\beta_{q+1}H_{q+1})(\tilde x_{0:q+1}) S_{1:q}(\tilde x_{0:q})\ .
%\end{eqnarray*}
%
      %}
%%%%%%%%%%%%%%%%%%%%%%%%%%%%%%%
\end{rem}
Now, let us introduce the sequence of probability measures $(\eta_k)$ defined by normalization of $(\gamma_k)_{k\geq 1}$ 
\begin{equation}
\label{eq:nu}
\eta_k(\varphi):=\frac{\gamma_k(\varphi)}{\gamma_k(\mathbf{1})}=\frac{\E[\varphi (\check X_{0:k}) G_{0:k-1}(\check X_{0:k-1})]}{\E[ G_{0:k-1}(\check X_{0:k-1})]}\ ,\quad \textrm{for any}\ k\geq 0\ ,
\end{equation}
where $\mathbf{1}$ denotes the function which takes the unique value $1$.
Observing that for $k\geq 1$, $\gamma_k(\mathbf{1}) =\gamma_{k-1}(G_{k-1})$,  we obtain by recurrence
\begin{eqnarray}
\label{eq:prod}
\gamma_k(\varphi)&=&\eta_k(\varphi)\gamma_k(\mathbf{1})\nonumber \\
&=&\eta_k(\varphi)\gamma_{k-1}(G_{k-1})\nonumber \\
&=&\eta_k(\varphi)\eta_{k-1}(G_{k-1})\cdots \eta_0(G_0)\ .
\end{eqnarray}
%%%%%%%%%%%%%%%%%%%%%%%%%%
    %{\bf  \\ XW
      %Au dessus tu va jusqu'a $\eta_1(G_1)$ ou bien  $\eta_0(G_0)$ qui doit etre egal a 1 je crois. En dessous tu vas juqu'a 0 \\
      %}
%%%%%%%%%%%%%%%%%%%%%%%%%%
As announced, we have replaced the expectation of a product of functions by the product of expectations of functions, since for any $n\geq 1$   
$$
u(t_0,x_0)=\gamma_n(\varphi_n) =\E[\varphi_n (\check X_{0:n})]=\eta_n(\varphi _n)\eta_{n-1}(G_{n-1})\cdots \eta_0(G_0)\ .
$$

Our objective is now to approximate the sequence of probability measures $(\eta_k)_{k\geq 0}$ by a sequence of empirical measures $(\eta^N_k)_{k\geq 0}$ based on a system of $N$ particles to finally end up with an approximation of the type 
$$
u(t_0,x_0)\approx \eta^N_n(\varphi_n)\eta^N_{n-1}(G_{n-1})\cdots \eta^N_0(G_0)\ .
$$

\subsection{The particle approximation scheme}
%%%%%%%%%%%%%%%%%%%%%%%%%%%
The sequence of approximating measures $(\eta_k^N)_{k\geq 0}$ will be defined by mimicking the dynamics of $(\eta_{k})_{k\geq 0}$. Hence, we begin by describing this recursive dynamics. \\
First let $K_k$ denote the transition kernel of the path valued Markov chain $(X'_k:=\check X_{0:k})$ from $k-1$ to $k$ for any integer $k\geq 1$. Recall that $K_k$ can be considered both as an integral operator on the space of measurable functions defined on $E_k$ and on the space of finite measures, $\mathcal{M}(E_{k-1})$, such that 
\begin{itemize}
	\item for any measurable test function  $f_k$ defined on $E_k$, $K_k(f_k)$ is a measurable function defined on $E_{k-1}$ such that for any $x'_{k-1}\in E_{k-1}$
$$
K_k(f_k)(x'_{k-1})=\E[f_k(X'_k)\,\vert X'_{k-1}= x'_{k-1}]=\int_{y'_{k}\in E_k} K_k(x'_{k-1},dy'_{k}) f_k(y'_{k})\ ,
$$
	\item for any finite measure $m_{k-1}$ on $E_{k-1}$, $m_{k-1}K_k$ is a finite measure on $E_k$ such that for any $x'_k\in E_k$  
	$$
	(m_{k-1} K_k)(dx'_{k})=\int_{y'_{k-1}\in E_{k-1}} m_{k-1}(dy'_{k-1})K_k(y'_{k-1},dx'_{k})\ .
	$$
	In particular, let $\mu_k$ denote the probability law underlying the random variable $X'_k:=\check X_{0:k}$ (we will often write $\mu_k=\mathcal{L}(X'_k)$), for any $k\geq 0$. Then  observe that $\mu_{k}K_{k+1}=\mu_{k+1}$ the probability law  of $X'_{k+1}:=\check X_{0:k+1}$.  Besides, notice that if $\check K_k$ denotes the transition kernel of the Markov chain $(\check X_{k})$ from $k-1$ to $k$, then  the transition kernel $K_k$ is obtained as the following cartesian product, for any $(y'_{k-1},x'_k):=(y_{0:k-1},dx_{0:k})\in E_{k-1}\times E_k$
	$$
	K_k(y'_{k-1},dx'_{k})=K_k(y_{0:k-1},dx_{0:k})=\delta_{y_{0:k-1}}(dx_{0:k-1})\times \check K_k(y_{k-1},dx_k)\ .
	$$
\end{itemize}
%    {\bf XW A priori $Q_h$ defini sur $E_k$ donc
 %$$
%Q_k(h_k)(\tilde x_{0:k-1})=\E[h(\tilde X_k)\,\vert \tilde X_{k-1}= \tilde x_{0:k-1}]\ ,
%$$
%A quoi ca sert ? On veut une $Q$ pour un $h$ particulier qui correspond au noyau de transition de la chaine de markov.
%D ailleurs apres, $Q_k$ n'est plus utilise sur une fonction ($m_{k-1} Q_k$..)
%Meme remarque
	%$$
	%(m_{k-1} Q_k)(d\tilde X_{0:k})=\int_{\tilde y_{0:k-1}} m_{k-1}(d \tilde y_{0:k-1})Q_k(\tilde y_{0:k-1},d \tilde X_{0:k})\ .
%$$
      %}
			Now we can describe the dynamics of $(\eta_k)_{k\geq 0}$ with $k$. 
For any real valued test function $f_k$ defined on $E_k$, the following identities holds 
\begin{eqnarray*}
\eta_{k}(f_k)&:=&\frac{\gamma_k(f_k)}{\gamma_k(\mathbf{1})}\\ \\
&=&\frac{\mu_{k}(f_k G_{1:k-1})}{\mu_{k}(G_{1:k-1})} \ ,\quad \textrm{where}\ \mu_k:=\mathcal{L}( X'_{k})=\mathcal{L}(\check X_{0:k})\ ,\textrm{and}\ G_{1:k}:=\prod_{p=1}^k G_p\\  \\
%&=& \frac{\mu_{k}(h_k G_{k-1} G_{1:k-2})}{\mu_{k}(G_{k-1} G_{1:k-2})} \\  \\
&=& \frac{ \mu_{k-1} (K_k(f_k) G_{1:k-1})}{\mu_{k-1}(G_{1:k-1})}\quad\textrm{by the tower property of conditional expectation} \\ \\
&=&\frac{\gamma_{k-1}(K_k(f_k)G_{k-1})}{\gamma_{k-1}(G_{k-1})} \quad\textrm{by definition~\eqref{eq:gamma} of } \gamma_{k-1}\\ \\
&=&\frac{\eta_{k-1}(K_k(f_k)G_{k-1})}{\eta_{k-1}(G_{k-1})} \quad\textrm{by dividing the numerator and denominator by}\  \gamma_{k-1}(\mathbf{1})\\
&=&((G_{k-1}\cdot \eta_{k-1})K_k)(f_k)\ ,
\end{eqnarray*}
where the $\cdot$ sign denotes the projective product between a non-negative function $G$ defined on $E$ and a non-negative measure $\mu \in \mathcal{M}^+(E)$ returning the probability measure $G\cdot \mu$ such that
\begin{equation}
\label{eq:dot}
(G\cdot \mu )(dx):=G(x)\mu (dx)/\mu(G)\ .
\end{equation}
    %{ \bf XW
      %Encore difficle a comprendre  :$ \mu_k:=\mathcal{L}(\tilde X_{0:k})$ pas clair comme notation.
      %En fait je comprends pas bien les deux premieres lignes
    %}
Hence, one can describe the evolution from $\eta_{k-1}$ to $\eta_k$ into two steps 
\begin{equation}
\label{eq:nuEvol}
\eta_{k-1}\xrightarrow[]{\textrm{Correction}}
\hat{\eta}_{k-1}:=G_{k-1}\cdot \eta_{k-1}
\xrightarrow[]{\textrm{Evolution}}
\eta_k:=\hat{\eta}_{k-1}K_k
\ ,
\end{equation}
%{\bf XW Bon je suis casse pied mais c'est pas tres comprehensible pour juste exprimer le fait qu'on fait du resampling}
In other words, the sequence of probability measures $(\eta_k)$ satisfies the following recursion
\begin{equation}
\label{eq:nu_rec}
\left \{
\begin{array}{l}
\eta_0= \mu_0\ ,\quad \textrm{where}\ \mu_0:=\mathcal{L}(X'_{0})=\mathcal{L}(\check X_0)  \\ \\
\hat{\eta}_{k}:=G_{k}\cdot \eta_{k}\ , \quad \textrm{for all}\quad  1\leq k\leq n\ , \\ \\
\eta_{k+1}= \hat{\eta}_{k}K_k\ , \quad \textrm{for all}\quad  1\leq k\leq n\ .
\end{array}
\right .
\end{equation}
%This nonlinear evolution can be summarized by 
%\begin{equation}
%\label{eq:phiEvol}
%\eta_{k+1}=\Phi_{k+1}(\eta_k)\ ,
%\end{equation}
%where the nonlinear operator, $\Phi_k$ defined on the space of sigma-finite and non-negative measures $\mathcal{M}^+(E_{k-1})$ and taking values in $\mathcal{M}^+(E_k)$ is such that
%\begin{equation}
%\label{eq:Phi}
%\Phi_k(m_{k-1}):=(G_{k-1}\cdot m_{k-1})K_k\ ,\quad \textrm{for any}\quad m_{k-1} \in \mathcal{M}^+(E_{k-1})\ .
%\end{equation}

%\subsubsection{Particle algorithm}
%%%%%%%%%%%%%%%%%%%%%%%%%%%%%%%%%%%

An Interacting Particle System  will be used to approximate the sequence of probability measures $(\eta_{k})_{0\leq k\leq n}$ by a sequence of empirical probability measures $(\eta^N_{k})_{0\leq k\leq n}$, such that for all $1\leq k\leq n$, $\eta^N_{k}$ is associated with an $N$-samples $(\xi^{1,N}_{k},\cdots, \xi^{N,N}_{k})$ approximately distributed according to $\eta_{k}$. To simplify the notation, we will often drop the exponent $N$ and write $(\xi^{i}_k)_{i=1,\cdots N}$ instead of $(\xi^{i,N}_k)_{i=1,\cdots N}$. 
The recursive evolution described by~\eqref{eq:nu_rec} is approximated by the following dynamics: 
\begin{equation}
\label{eq:nu_part}
\left \{
\begin{array}{l}
\eta^N_0=\mu_0\\ \\
\hat{\eta}^N_{k}=G_{k}\cdot \eta^N_{k} \ , \quad \textrm{for all}\quad 1\leq k\leq n\\ \\
\eta^N_{k+1}=S^N(\hat{\eta}^N_{k}K_{k})\ , \quad \textrm{for all}\quad 1\leq k\leq n\ ,
\end{array}
\right .
\end{equation}
where $S^N(\mu)$ denotes the empirical measure associated to an $N$-sample $(\xi^1,\cdots , \xi^N)$ i.i.d. according to $\mu$, that is 
$$
S^N(\mu)=\frac{1}{N}\sum_{i=1}^N\delta_{\xi^i}\ ,\quad \textrm{where}\ (\xi^1,\cdots ,\xi^N)\ \textrm{i.i.d.}\ \,\sim \, \mu\ . 
$$
Hence, the algorithm proceeds as follows. Recalling that $G_0=\mathbf{1}$,  we initiate the algorithm by generating $N$ i.i.d. random variables $(\xi^1_1,\cdots , \xi^N_1)$ according to $\mu_0$, then we set  
\begin{equation}
\label{eq:eta1}
{\displaystyle \eta^N_1=S^N(G_0\cdot\mu_0)=S^N(\mu_0)=\frac{1}{N}\sum_{i=1}^N \delta_{\xi^i_1}}\ .
\end{equation}
The evolution of the discrete measures, $(\eta^N_{k})_{0\leq k\leq n}$, (where $N$ denotes the size of the particle system) between two iterations $k$ and $k+1$, consists into three steps: 
\begin{enumerate}
	\item \textbf{Weighting}: each particle is weighted according to the value of the current potential function $G_{k}$. For all $i\in \{1,\cdots, N\}$, we compute ${\displaystyle \omega_{k}^{i} =\frac{G_{k}(\xi_{k}^{i})}{\sum_{j=1}^N G_{k}(\xi_{k}^{j})}}$ and we set  $\hat{\eta}_{k}^N= \displaystyle 
\sum_{i=1}^N \omega_{k+1}^{i}\, \delta_{\xi_{k+1}^{i}}$.  
	\item \textbf{Selection}:   
$N$ i.i.d. random variables $(\hat{\xi}_{k}^{1},\cdots, \hat{\xi}_{k}^{N})$ are generated according to the weighted discrete probability distribution $\hat{\eta}_{k}^N = \displaystyle 
\sum_{i=1}^N \omega_{k}^{i}\, \delta_{\xi_{k}^{i}}$.  More specifically, for all $i\in \{1,\cdots, N\}$, an index $I\in \{1,\cdots, N\}$ is generated independently  with probability $\mathbb{P}(I=j)=\omega_{k}^{j}$ and we set $\hat{\xi}_{k}^{i}={\xi}_{k}^{I}$. 
	\item \textbf{Mutation}:   
Each selected particle evolves independently according to the dynamics $K_{k+1}$. 
This produces a new particle system $(\xi_{k+1}^{1},\cdots, \xi_{k+1}^{N})$. 
More specifically, for all $i\in \{1,\cdots, N\}$, we generate independently $\xi^i_{k+1}$ according to the conditional distribution $\mathcal{L}(X'_{k+1} \vert {X}'_{k}=\hat \xi^i_k)$, then we set 
\begin{equation}
\label{eq:etak}
\eta_{k+1}^N= \displaystyle \frac{1}{N}\,\sum_{i=1}^N \delta_{\xi_{k+1}^{i}}\ .
\end{equation}
\end{enumerate}
For all $k\geq 1$, let us introduce $\gamma_k^N$, the particle approximation of $\gamma_k$ based on $\eta^N_k$ defined by recursion~(\ref{eq:nu_part}) and such that for any real valued measurable test function $f_k$ defined on $E_k$, 
\begin{equation}
\label{eq:gammapN}
\gamma^N_k(f_k)=\eta^N_k(f_k) \prod_{0\leq p\leq k-1} \eta^N_p(G_p)\ .
\end{equation}
We begin by stating a Lemma that will be crucial to prove the convergence of our new estimator.
\begin{lem}
\label{lem:gammaN}
Let $(X'_n)_{n\geq 0}$ be a Markov chain (with initial distribution $\mu_0$ and transition kernel $K_k$) defined on a sequence of measurable spaces $(E_n, \mathcal{E}_n)_{n\geq 0}$ and $(G_n)_{n\geq 0}$ be a sequence of positive measurable functions defined on $(E_n, \mathcal{E}_n)_{n\geq 0}$ such that there exists a finite constant $A\geq 2$ such that 
\begin{equation}
\label{eq:assG}
\sup_{x'_{0}\in E_{0}} G_0(x'_0)\leq A\ ,\quad \textrm{and}\quad \sup_{x'_{p-1}\in E_{p-1}} \E[G^2_{p}(X'_{p})\vert X'_{p-1}= x'_{p-1}]\leq A\ ,\quad\textrm{for any}\  p\geq 1 \ .
\end{equation}
 We consider the sequence of Feynman-Kac measures $(\gamma_n)$ such that for any measurable real valued function $f_n$ defined on $E_n$, 
\begin{equation}
\label{eqFK}
\gamma_n(f_n):=\E[f_n(X'_{n})\prod_{k=0}^{n-1} G_k(X'_{k})]\ .
\end{equation}
Let $(\gamma_n^N)$ be a sequence of particle approximation measures of $(\gamma_n)$ defined similarly as in~\eqref{eq:gammapN}, with $(\eta_p^N)_{0\leq p}$ defined by~\eqref{eq:nu_part}. 
For a given $n\geq 1$, let us consider a real valued measurable function $f_n$ defined on $E_n$ such that there exists a finite positive constant $B$ such that 
\begin{equation}
\label{eq:assf}
\sup_{x'_{p-1}\in E_{p-1}}\vert \E[f^2_n(X'_n)G^2_{p:n-1}(X'_{p:n})\vert X'_{p-1}=x'_{p-1}]\leq B\quad \textrm{for any}\  p=1,\cdots n \ .
\end{equation}
Then the particle approximation $\gamma_n^N(f_n)$ is unbiased with finite variance, more precisely
\begin{equation}
\label{eq:BiasVargammaLem}
\E[\gamma^N_n(f_n)]=\gamma_n(f_n)\ ,\ \textrm{and}\quad \E[\big (\gamma^N_n(f_n)-\gamma_n(f_n)\big )^2]\leq 2B\frac{A^{n+2}}{N}\quad \textrm{for}\  N\geq A^{n+1}\ .
\end{equation}
\end{lem}
The proof of this Lemma relies on the formalism developed in the reference books~\cite{DelMoral,DelMoral2}. However, we had to carry out an original proof to take into account our specific framework where the potential functions $G_k$ are unbounded which is not considered to our knowledge in the existing literature. The proof is placed in the Appendix~\ref{sec:Appendix}. 

We are now in a position to state the main result of this section.
%
%\subsubsection{Convergence of the resampling estimator}
%%%%%%%%%%%%%%%%%%%%%%%%%%%%%%%%%%%%%%%%%%%%%%%%%%%%%%%%
%
\begin{thm}
\label{thm:gammaN}
Suppose that Assumptions~\ref{ass:unique2},~\ref{ass:Bound}
and~\ref{ass:LH} are satisfied. 
For any $n\geq 1$, the resampling estimator $\gamma_n^N(\varphi_n)$ defined by~\eqref{eq:gammapN} is unbiased with finite variance. More precisely, 
\begin{equation}
\label{eq:BiasVargamma}
\E[\gamma^N_n(\varphi_n)]=u(t_0,x_0)\ ,\quad \textrm{and}\quad \E[\big (\gamma^N_n(\varphi_n)-u(t_0,x_0)\big )^2]\leq \frac{C^{n+2}}{N}\quad \textrm{for}\ N\geq C^{n+1}\ ,
\end{equation}
where $(\varphi_n)_{n\geq 1}$ is a sequence of real valued functions defined on $E_n$ by~\eqref{eq:varphi} and $C$ is a constant depending only on the characteristics of the problem ($T$, the bounds or Lipschitz constants related to $g$, $b$, $\sigma$, $a$).   
\end{thm}
\begin{rem}
\label{rem:gamman}
\begin{enumerate}
	\item Computing $\gamma_n^N(\varphi_n)$ reduces to compute the following product of empirical means 
\begin{eqnarray*}
\gamma_n^N(\varphi_n)&=&\eta_n^N(\varphi_n)\eta_{n-1}^N(G_{n-1})\cdots \eta_0^N(G_0)\\
&=&\left (\frac{1}{N}\sum_{i=1}^N \varphi_n(\xi^i_n)\right )\left (\frac{1}{N}\sum_{i=1}^N G_{n-1}(\xi^i_{n-1})\right )\cdots \left (\frac{1}{N}\sum_{i=1}^N G_{0}(\xi^i_{0})\right )\ ,
\end{eqnarray*}
where $(\xi^i_k)_{1\leq i\leq N}$ is the particle system at the $k$th iteration of the algorithm as stated by~\eqref{eq:etak}. 
This in particular requires to compute $\varphi_n(\xi^i_n)$ for each particle of the final particle system $(\xi^i_n)_{i=1,\cdots N}$. Recalling Remark~\ref{rem:varphin}, this may require to compute a conditional expectation. 
In practice, one chooses $n$ large enough such that most of particles have already reached time $T$ after $n$ iterations 
implying that for most particles $\varphi_n(\xi_n^i)$ can be computed explicitly. 
In the rare cases of particles $\xi^i_n$ that have not reached yet time $T$, the computation of $\varphi_n(\xi_n^i)$ that should normally require to compute a conditional expectation is approximated by one simulation according to 
$$
\mathcal{L}((\beta_{N_T+1}H_{N_T+1}S_{0:N_T+1})(\check X_{0:N_T+1})G_{n:N_T}(\check X_{0:N_T})\,\vert \check X_{0:n}=\xi^i_{n})\ .
$$
Notice that it would be interesting to consider the estimator 
$$
\gamma_{n_N}^N(\varphi_{n_N})\ ,\quad \textrm{with}\quad n_N=\inf \{n\,\vert\,\xi^i_n  \ \textrm{has reached $T$ for all}\ i=1,\cdots N\}\ .
$$
This will be left for future work. 
\item Another approach to avoid this problem would consists in doing the resampling procedure only on the space variables. First simulate a sequence of random switching times $(T_1,\cdots ,T_{N_T})$ and conditionally to this time mesh run an interacting particle system on the Markov chain $\bar X$~\eqref{eq:Xk}. The estimator would then be given as an empirical mean of the resampling estimates over i.i.d. time meshes. 
\end{enumerate}
\end{rem}
\begin{proof}
Theorem~\ref{thm:gammaN} is a direct consequence of Proposition~\ref{prop:ugamma} stating that $\gamma_n(\varphi_n)=u(t_0,x_0)$ and of Lemma~\ref{lem:gammaN} after having verified that there exists a finite positive constant $C$ for which the bounds~\eqref{eq:assG} and~\eqref{eq:assf} are verified. Observe that~\eqref{eq:assG} is automatically implied by~\eqref{eq:BoundGk}. 
Let us consider~\eqref{eq:assf}, similarly to the proof of Proposition~\ref{prop:finVar} one obtains
\begin{equation}
\label{eq:sum}
\begin{array}{l}
\E[\varphi^2_n(\check X_{0:n})G^2_{p:n-1}(\check X_{p:n})\vert \check X_{0:p-1}=\check x_{0:p-1}]
\\
=
\sum_{q=0}^\infty 
\E[\varphi^2_n(\check X_{0:n})G^2_{p:n-1}(\check X_{p:n})\vert \check X_{0:p-1}=\check x_{0:p-1}, N_T=q]\mathbb{P}(N_T=q)
\ .
\end{array}
\end{equation}
Now considering the general term of this sum for $q\geq p\geq 2$
%using inequality~\eqref{eq:cPk} yields
$$
\begin{array}{l}
\E[(\beta_{N_T+1}^2H^2_{N_T+1})(\check X_{0:N_T+1})\prod_{k=p}^{N_T}G^2_{k}(\check X_{0:k})\,\vert \check X_{0:p-1}=\check x_{0:p-1}, N_T=q]
\\ 
= 
\E[\beta^2\frac{1}{c_{p-1}(\check X_{0:p-1})}\frac{1}{(\delta T_{p-1})^{2(1-\kappa)}} \prod_{k=p-1}^{q-1} P^2_{k+1}\,\vert \check X_{0:p-1}=\check x_{0:p-1}, N_T=q]\\
\leq 
C\E[(\delta T_q)^{2(1-\kappa)}\frac{1}{c_{p-1}(\check X_{0:p-1})}\frac{1}{(\delta T_{p-1})^{2(1-\kappa)}} c_{q}(\check X_{0:q})P^2_q\prod_{k=p-1}^{q-2} P^2_{k+1}\,\vert \check X_{0:p-1}=\check x_{0:p-1}, N_T=q]\ ,
\end{array}
$$
where $C$ is a constant that may change from line to line.
Recalling~\eqref{eq:cPkbis} finally gives 
$$
\begin{array}{l}
\E[(\beta_{N_T+1}^2H^2_{N_T+1})(\check X_{0:N_T+1})\prod_{k=p}^{N_T}G^2_{k}(\check X_{0:k})\,\vert \check X_{0:p-1}=\check x_{0:p-1}, N_T=q]
\\
\leq 
C^{q-p+1}\E[ (\delta T_q)^{2(1-\kappa)}\frac{c_{p-1}(\check X_{0:p-1})}{c_{p-1}(\check X_{0:p-1})}\frac{1}{(\delta T_{p-1})^{2(1-\kappa)}}\prod_{k=p-1}^{q-1} \frac{(\delta T_{k})^{2(1-\kappa)}}{(\delta T_{k+1})^{2((1-\alpha)\vee\frac{1}{2})}}
\,\vert N_T=q]\\
\leq 
C^{q-p+1}\ .
\end{array}
$$
We proceed similarly when $p=1$. 
%%%%%%%%%%%%%%%%%%%%%%%%%%%%%%%%%%%%%%%%%%%%%%%%%%%%
%%\\
%{ \bf XW  With my weights...\\
  %
%Now considering the general term of this sum, using \eqref{eq:ab}  and  an equation similar to  \eqref{eq::recursion}
%$$
%\begin{array}{l}
%\E[(\beta_{N_T+1}^2H^2_{N_T+1})(\tilde X_{0:N_T+1})\prod_{k=p}^{N_T}G^2_{k}(\tilde X_{0:k})\,\vert \tilde X_{0:p-1}=\tilde x_{0:p-1}, N_T=q]
%\\
%= 
%\E[\beta^2  \frac{1}{\delta T_{p-1}^{2\alpha} ( \sqrt{\delta T_{p-1}}+\Vert \delta a_{p-1} \Vert + \Vert \delta b_{p-1} \Vert )}\prod_{k=p}^{q} P^2_{k}\,\vert \tilde X_{0:p-1}=\tilde x_{0:p-1}, N_T=q] \\
%\le  C^{q-p+1} \E[(\delta T_q)^{2(1-\kappa)} \frac{C_{p-1}}{\delta T_{p-1}^{2\alpha} ( \sqrt{\delta T_{p-1}}+\Vert \delta a_{p-1} \Vert + \Vert \delta b_{p-1} \Vert )}  \prod_{k=p}^{q-1}\frac{(\delta T_k)^{2(1-\kappa)}}{\delta T_{k+1}}\,\vert \tilde X_{0:p-1}=\tilde x_{0:p-1},\vert N_T=n]\ \\
%\le C^{q-p+1} (\delta T_{p-1}^{2-2\kappa -2 \alpha})   \E[ \prod_{k=p}^{q} (\delta T_k)^{1-2 \kappa}\,\vert \tilde X_{0:p-1}=\tilde x_{0:p-1} ] \\
%\le C^{q-p+1}
%\end{array}
%$$
%where  $C$ is a constant that may change from line to line.
%}
%%\\
%%%%%%%%%%%%%%%%%%%%%%%%%%%%%%%%%%%%%%%%%%%%%%%ù
We conclude by observing that the sum~\eqref{eq:sum} is finite by the same argument as in the proof of Proposition~\ref{prop:finVar}. 
\end{proof}

\section{Numerical simulations}
%-------------------------------------------
\label{sec:simu}

In this section, we begin by an empirical analysis of complexity then we analyse and compare the performances of the three approaches described previously 
\begin{enumerate}
\item Switching Monte Carlo method with exponential switching times;
\item Switching Monte Carlo method with gamma switching times (with parameter $\kappa\leq 1/2$);
\item Resampling and Switching Monte Carlo method with gamma switching times  (with parameter $\kappa\leq 1/2$). 
\end{enumerate}
On one test case, we compare numerically the Switching Monte Carlo method with gamma switching times with  the Euler Monte Carlo method.\\
First, we consider a simple example for which all assumptions of Proposition \ref{prop:finVar}  are satisfied. 
Then we consider simulations involving a more standard payoff function $g$ occuring in finance (corresponding to the call option) that does not fulfill Assumption \ref{prop:finVar}. However, this offers the opportunity to check the robustness of our approach out of theoritical assumptions.\\
In all cases, we consider 
\begin{itemize}
\item a drift coefficient  $ b(t,x)= -10 \vee (1-x) \wedge 10$,
\item an initial condition $x_0=1$,
\item a terminal time $T=1$.
\end{itemize}
The parameters of the switching time distributions is  $\lambda=0.4$ for the exponential distribution.  Even if the exponential distribution gives a theoretical infinite variance (in cases we consider here), the numerical variance observed is finite so
it is interesting to compare the results obtained by the gamma distribution and the exponential distributions. \\
To implement efficiently the different methods on a computer using many cores (96 on our computer), we allocate $N$ particles to each core such that the total number of particles used is  $n_{\textrm{part}}=96N$. When resampling is used, a resampling estimator $\gamma_p^{N,j}(\varphi_p)$ is simulated independently on each core $j=1,\cdots ,\,96$ and we return the average estimator~: $\frac{1}{96}\sum_{j=1}^{96} \gamma_p^{N,j}(\varphi_p)$. 
%the variance reduction is achieved  on each core (so with $\frac{n}{96}$ particles) and the results on all the cores are averaged to get a first estimation.
Then the procedure is repeated independently for $1000$ estimations, so as to approximate empirically the expectation and the variance of each estimator by the empirical average and variance computed on the $1000$ estimates.\\
The whole procedure is then repeated for different values of $n_{\textrm{part}}=4^qn_0$ from $q=0$ to $q=5$, with   $n_0=10^4$.
We reported on the graphs the evolution of the estimator expectation as a function of $\log (n_{\textrm{part}})$ and the related standard deviation is represented on $\log$-$\log$ graphs. On each figure devoted to the standard deviation, the theoretical decrease at a rate $1/(n_{\textrm{part}})^{1/2}$ is represented by the plot of a line  with slope $-0.5$.

\subsection{Complexity analysis}
The Switching Monte Carlo method requires to simulate, for each trial,  a random number of time steps, $N_T$, before reaching $T$. In order to analyze the impact of the parameters $\kappa$ and $\theta$ on the complexity of the algorithm, we consider  $\hat N_T:=\E[N_T]$ as a function of $(\kappa,\theta)$. As we couldn't derive any analytical approximation, we have computed  a numerical
estimate which is reported on Figure \ref{figNT} for different values of $\kappa$ and $\theta$.
\begin{figure}[H]
\centering
\includegraphics[width=12cm]{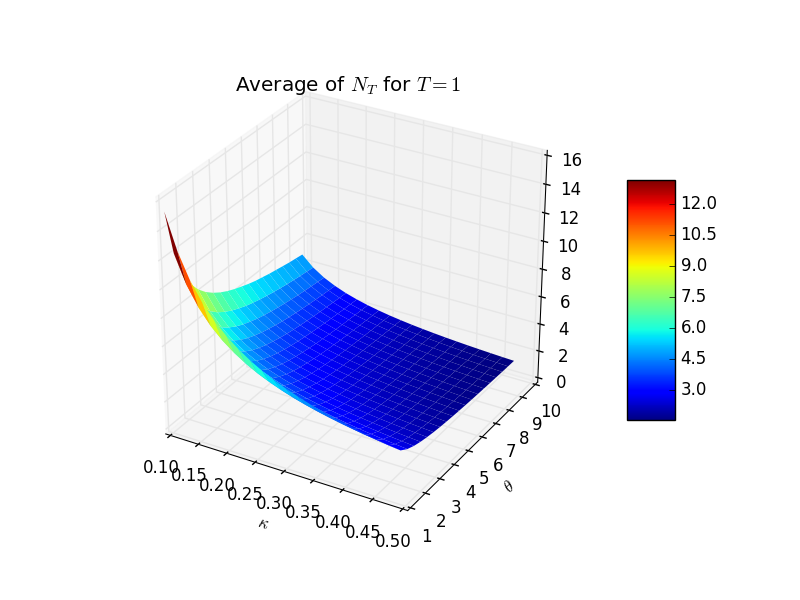}
\caption{$\hat N_T$ for different values of $\kappa$ and $\theta$ and $T=1$.}
\label{figNT}
\end{figure}
One can observe that the expected number of time steps increases as $\theta$ or $\kappa$ decreases. More precisely,  $\hat N_T(\kappa,\theta)$ can be  accurately estimated by the following  polynomial  approximation:
\begin{eqnarray*}
\hat N_T(\kappa, \theta) =15.84   -1.63 \theta  -46.16   \kappa +  46.36 \kappa^2 +  1.47 \theta \kappa\ ,
\end{eqnarray*}
for $\kappa \in [0.2, 0.5]$ and  $\theta \in [1,10]$, recalling that we are only interested by values of $\kappa\leq 1/2$. 
Besides, at each switching time of each trial the computational complexity is given by
\begin{eqnarray*}
C_{switch}(\kappa,\theta) + d (  C_{Gauss} +c_2)  + c_1 d^{2.3727} 
\end{eqnarray*}
where $c_1,c_2$ are given constants,  $C_{switch}(\kappa,\theta)$ is the complexity of generating the switching time (according to an exponential or a gamma law depending on the approach),  $C_{Gauss}$ is the complexity for generating a Gaussian r.v.,
and the term in $d^{2.3727}$ is the theoretical optimal cost  for $\sigma$ inversion by a $LU$ method.\\
The global complexity of the algorithm  without resampling for $n_{\textrm{part}}$ simulations is in high dimension:
\begin{eqnarray*}
  c  d^{2.3727} n_{\textrm{part}}
\end{eqnarray*}
\begin{rem}
  Notice that, based on our numerical tests,  the cost, $C_{switch}$ (in the gamma case),  of generating  a gamma r.v.  with a rejection method is on average between $300$ and $500$  floating operations, whereas the cost, $C_{Gauss}$, of  
  generating a Gaussian random variable requires around $10$ floating operations.
 Hence, for low dimension, the leading term corresponds to $C_{switch}$.
\end{rem}
With resampling, we have to add some operations independent of the dimension of the problem:  using the order statistics of the exponential law, we are able to generate some sorted uniformly distributed random variables that are used to select the particles during the selection step with a  cost linear with $n_{\textrm{part}}$.

\begin{rem}
  %The advantage of the method without resampling comes from the fact that the computational cost is strictly linear with the number of particles used: there are very few cache miss.
  The resampling method, by imposing to store the states of all simulations simultaneously, gives a computational cost (including the memory access time) increasing slightly  more than linearly with $n$ (see Figure \ref{FigTime} below). The advantage of the method without resampling comes from the fact that the memory access time is weaker so that the computational cost is strictly linear with the number of particles. 
\end{rem}

\subsection{An example with $g(x)=cos( x)$, $\sigma(t,x)= 0.5  + 0.2 (x^2 \wedge 1) $.}
In dimension 1,  we give on Figure~\ref{FigCaseA} the convergence observed with the exponential law and the gamma law with and without resampling.

\begin{figure}[H]
\centering
\includegraphics[width=7.5cm]{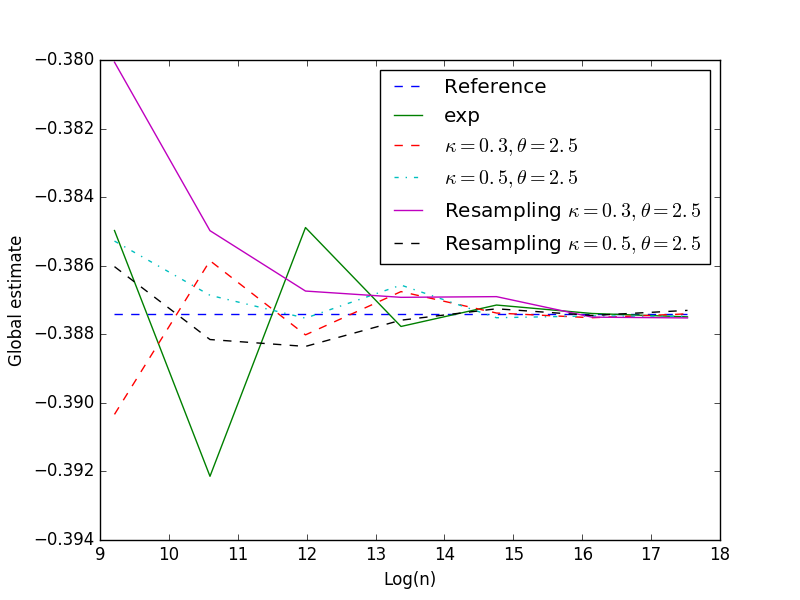}
\includegraphics[width=7.5cm]{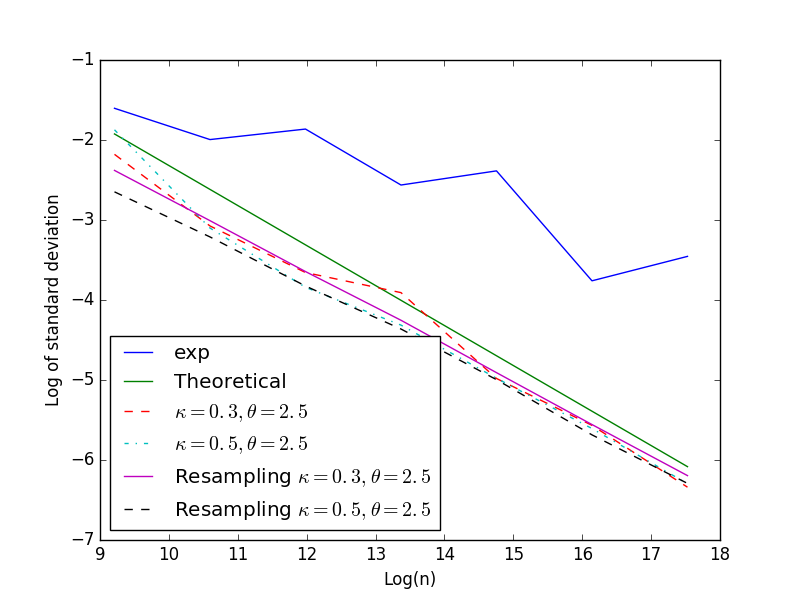}
\caption{Estimation and standard deviation observed for  case 1 (dimension 1).}
\label{FigCaseA}
\end{figure}
The method converges easily with the gamma laws.
Using the exponential distribution, the empirical standard deviation seems to decrease to zero but the rate $1/\sqrt{n}$ cannot be diagnoseds: this is the consequence of an infinite theoretical variance.\\
On this case, the resampling doesn't improve much the results because of the small variation of the $\sigma$ function.
With the gamma laws, the linear decay of the log of the standard deviation follows the theory with a slope equal to $-\frac{1}{2}$ with respect to $\log(n_{\textrm{part}})$.

\subsection{ One dimensional tests with $g(x)= (x-1)^{+}$.}
With this kind of $g$ function assumptions of Proposition \ref{prop:finVar} are not satisfied. We will show nevertheless that the method gives good results.
Because the variance of the results is closely related to the diffusion coefficient variation, we will consider various examples with $\sigma$ getting more and more space dependent.
\subsubsection{ $\sigma(t,x)= 0.5  + 0.2 (x^2 \wedge 1) $}
This second case shows some quite small variations of $\sigma$.  The reference value is $0.17466$. Evolution of the global estimate and the standard deviation are given
on Figure \ref{FigCase1} for gamma and exponential laws.
\begin{figure}[H]
\centering
\includegraphics[width=7.5cm]{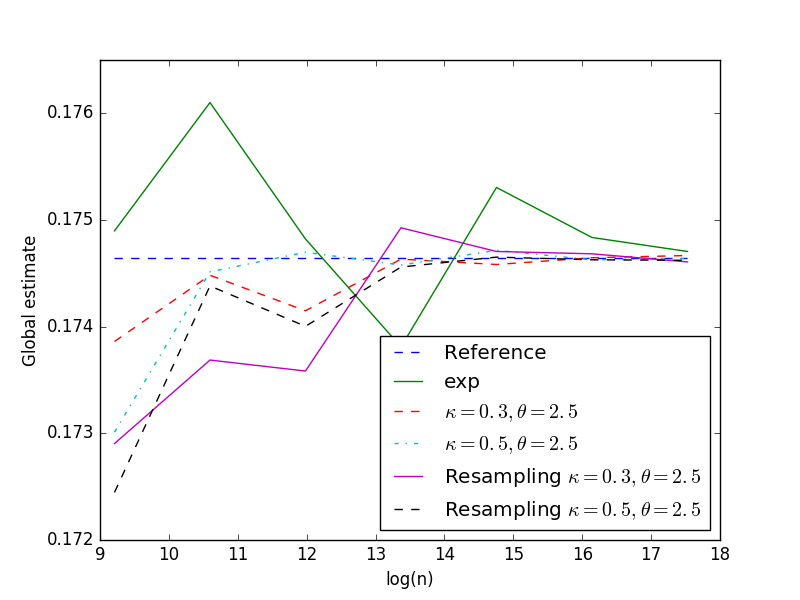}
\includegraphics[width=7.5cm]{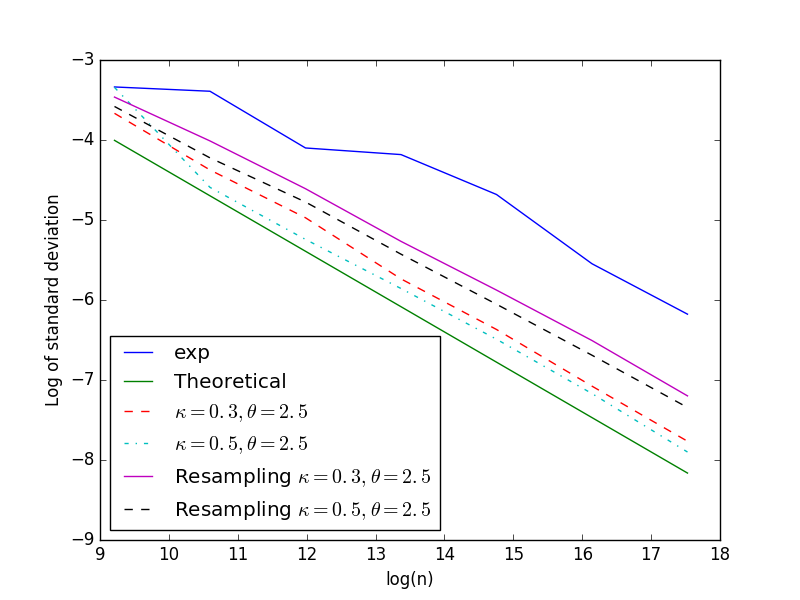}
\caption{Estimation and standard deviation observed for  case 2.}
\label{FigCase1}
\end{figure}
In this simple case easily converging as in the first case, resampling doesn't improve the standard deviation.

We are interested in comparing numerically the Switching Monte Carlo method (SMC) with the Euler Monte Carlo method (EMC). It is well known  \cite{Talay} that the error due the EMC method with a time discretization step, $h$, and a number of particles, $n_E$, can be decomposed into a bias term, 
$C_{E} h$, and a standard deviation term, ${S}/{\sqrt{n_{E}}}$. To achieve a fixed level of accuracy, $\varepsilon$ by balancing the two types of errors we quasi-optimally chose $h(\varepsilon): ={\varepsilon}/{(2\hat C_E)}$ and $ n_E(\varepsilon): = ({2 \hat S}/{\varepsilon})^2$, where  $\hat S$ and respectively $\hat C_E$ have been estimated 
 using a reference calculation with $30\times 10^6$ particles and ${1000}$ time steps and respectively using  $100$  time steps.
Using the empirical variance computed on $1000$ SMC estimators, we have estimated  the error, $\varepsilon$,  of the SMC method for different numbers of particles $n_{\textrm{part}}$. For each, error $\varepsilon$, we have used an  Euler Monte Carlo (EMC) with a time step $h(\varepsilon)$ and $n_E(\varepsilon)$ simulations to achieve the same error $\varepsilon$.  On Table~\ref{tabEuler}, we have reported for each error, $\varepsilon$, and for each approach, SMC or EMC, the associated computing time (Time), mean estimate (Mean), number of particles ($n_{\textrm{part}}$, $n_E$), and for the EMC approach, we have further reported the number of time steps. 
\begin{table}[H]
\caption{\label{tabEuler} Comparison of the Switching  Monte Carlo method (SMC) $\theta =2.5$, $\kappa=0.5$ without resampling  and  the Euler Monte Carlo method (EMC). The Reference value is $0.17466$. }
\begin{tabular}{|c|c|c|c|c|c|c|c|} \hline
 $\varepsilon$ & 8e-4 & 0.000398 & 0.000219 & 1.02 e-4  & 5.42e-5 & 2.67e-5 & 1.34e-5 \\ \hline\hline
  SMC Time& 0.65 & 2.11 & 8.88 & 34.76 & 125.85 & 502.49 & 1991.43 \\ \hline
SMC Mean& 0.173861 & 0.174482 & 0.17415 & 0.174633 & 0.174583 & 0.174646 & 0.17466 \\ \hline
  SMC $n_{\textrm{part}}$  & 10000 & 40000 & 160000 & 640000 & 2560000 & 1024e4 & 4096e4 \\ \hline\hline
 EMC Time& 0.11 & 0.93 & 5.62 & 55.5 & 369.67 & 3096.63 & 24429 \\\hline
 EMC Mean& 0.175225 & 0.174888 & 0.174615 & 0.174613 & 0.174674 & 0.174667 & 0.17465 \\ \hline
 EMC $n_E$  & 432025 & 1779528 & 5879543 & 27e6 & 96e6 & 396e6 & 1.579e9 \\ \hline
EMC $T/h$ & 145 & 295 & 536 & 1151 & 2167 & 4402 & 8763 \\ \hline
\end{tabular}
\end{table}
One can observe on  Table~\ref{tabEuler} that for errors $\varepsilon$ up to $10^{-5}$ the SMC method appears to be  slower than the Euler scheme by a factor between $1.5$ and $4$  whereas for very high precisions it begins to be more effective in the considered case. \\
On Figure \ref{FigTime}, we give the computional time of each method depending on the switching laws  and their parameters.
\begin{figure}[H]
\centering
\includegraphics[width=7.5cm]{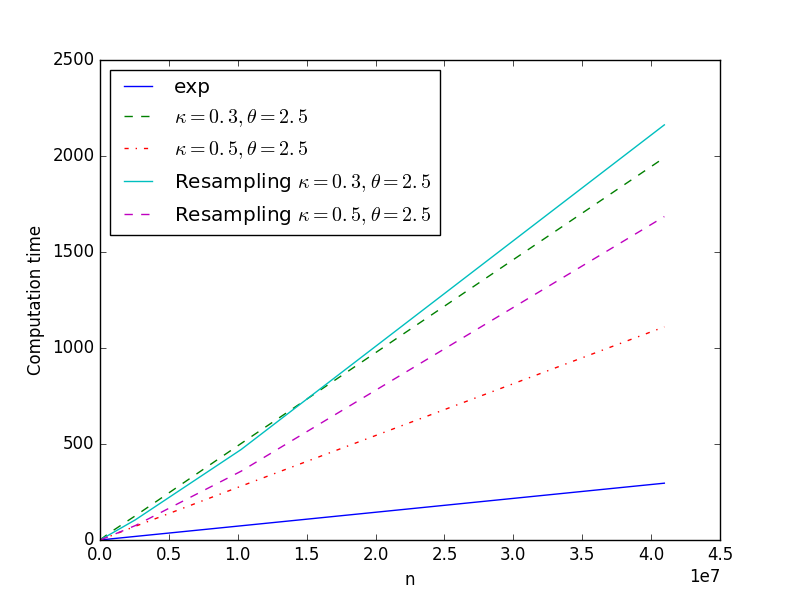}
\caption{Computational time as a function of the particle number, for different Switching Monte Carlo approaches (exponential, gamma, gamma with resampling) and parameters.}
\label{FigTime}
\end{figure}
One can observe on Figure~\ref{FigTime} that indeed  the computational time of the Switching Monte Carlo method is proportional to the number of particles. Of course, using a gamma distribution instead of an exponential one increases the number of switching times  and hence the computational time.
One can observe that with resampling,  the computational time increases slightly more than linearly with the number of particles. \\
On Figure \ref{KappaDependence}, for different levels of accuracy, we compute the computational time needed for the SMC method with gamma switching times to reach a given accuracy, depending on the $\kappa$ parameter ($\theta$ being fixed to $2.5$). 
\begin{figure}[H]
\centering
\includegraphics[width=8cm]{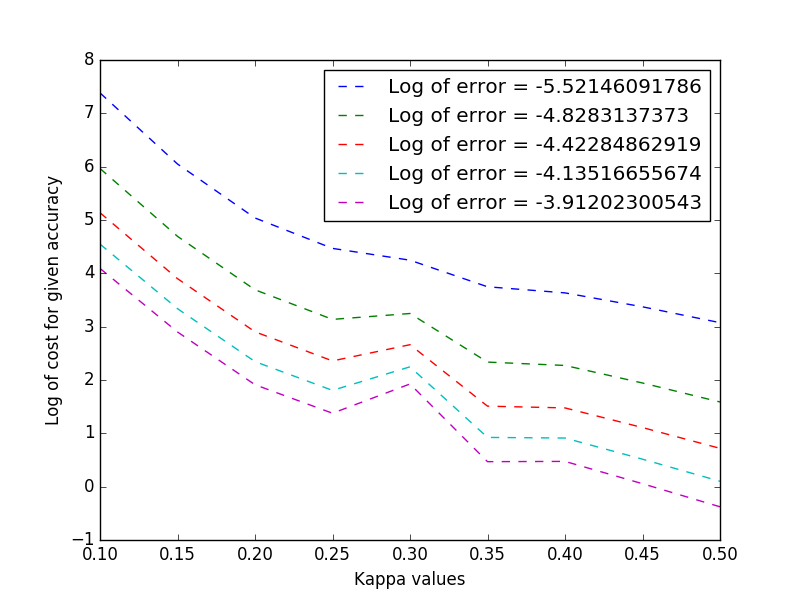}
\caption{Computional cost as a function of  $\kappa$ for different levels of accuracy.}
\label{KappaDependence}
\end{figure}
Clearly the optimal parameter is $\kappa=0.5$. In fact using a parameter $\kappa= 0.5$ doesn't improve much the accuracy of the result but the computional time required is smaller.

\subsubsection{ $\sigma(t,x)= 0.5  + 0.4 (x^2 \wedge 1) $.}
With this more difficult case, we report results obtained with $\theta=2.5$, $\kappa =0.3$, $\kappa=0.5$ with and without resampling for the gamma distribution and for the exponential
distribution.
The reference value is $0.21408$. Evolution of the mean estimate and the standard deviation are given
on Figure \ref{FigCase2}.
\begin{figure}[H]
\centering
\includegraphics[width=7.5cm]{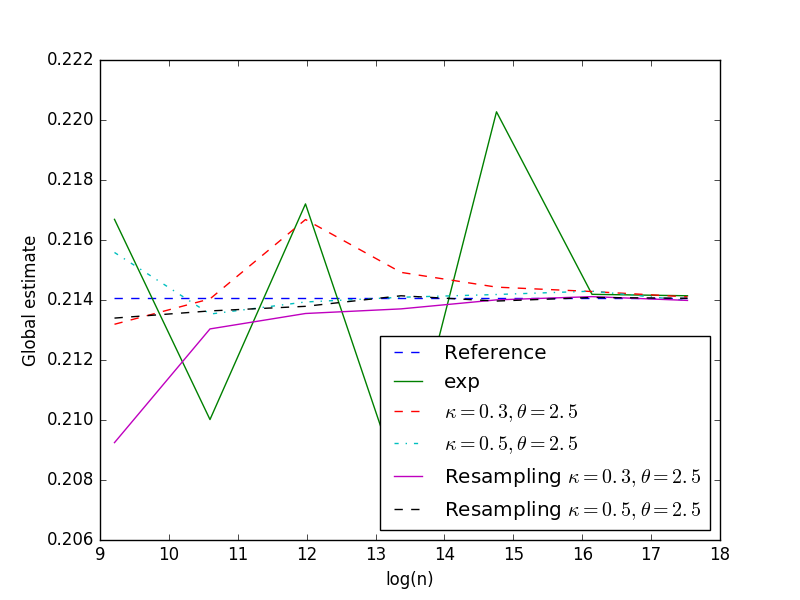}
\includegraphics[width=7.5cm]{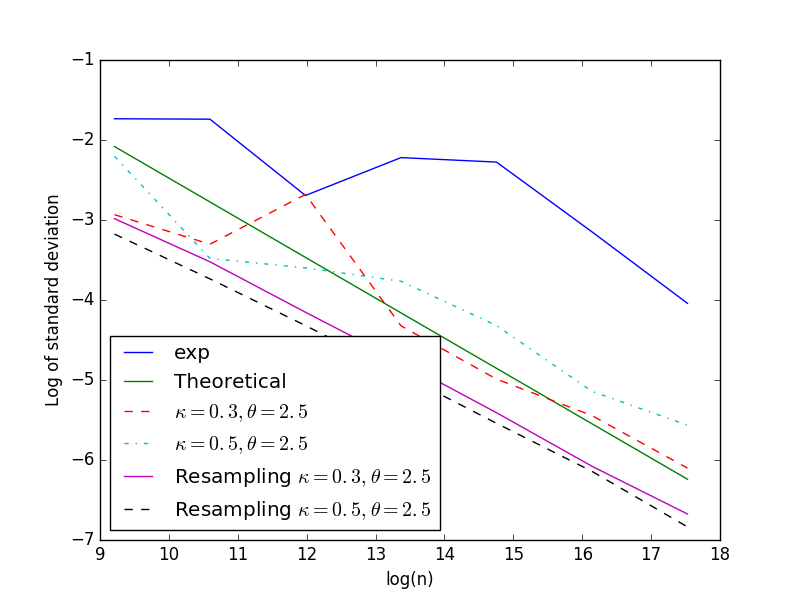}
\caption{Estimation and standard deviation observed for case 3.}
\label{FigCase2}
\end{figure}
With or without resampling, the standard deviation is decreasing steadily while using gamma distributions. As we quadruple the number of simulations, the standard deviation is roughly divided by two which is coherent with the theory. The convergence with the exponential law is erratic once more.

In the case of gamma distribution, the standard deviation with resampling is nearly half of the one without resampling clearly showing the interest in this method.
Results with $\kappa=0.5$ and $\kappa =0.3$ are very similar especially with resampling, but the number of jumps increases as $\kappa$ decreases and  the computational time is
nearly doubled with $\kappa=0.3$ indicating that the optimal choice is to take $\kappa=0.5$.

\subsubsection{ $\sigma(t,x)= 0.5 \vee x^2 \wedge 1$.}
This case is more difficult than the two first ones. 
The reference value is  $0.2100$. We keep $\theta=2.5$ and we use   $\kappa =0.3$ and $\kappa=0.5$.
Figure \ref{FigCase3}, reveals that without resampling the SMC method shows difficulties to converge, while the resampling SMC method is always converging with a standard deviation decreasing at the rate $1/\sqrt{n_{\textrm{part}}}$.  Besides the exponential case doesn't seem to converge at all.
\begin{figure}[H]
\centering
\includegraphics[width=7.5cm]{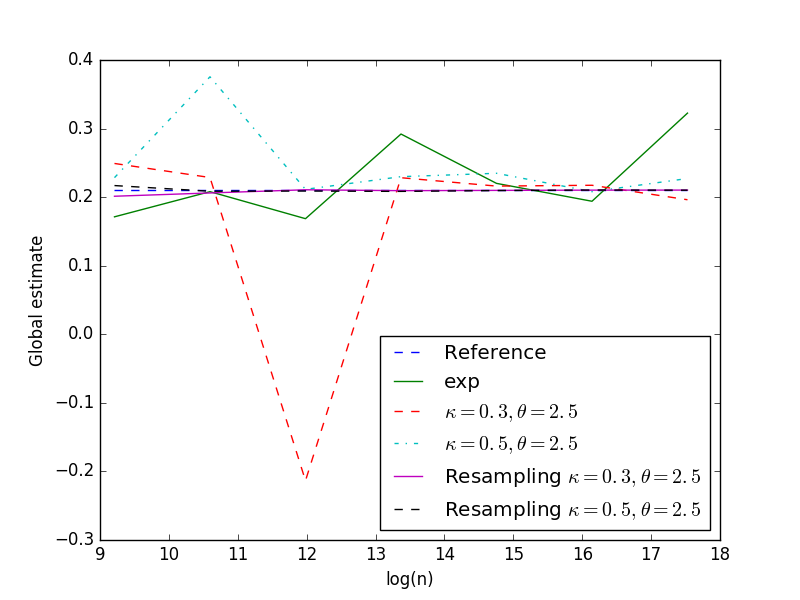}
\includegraphics[width=7.5cm]{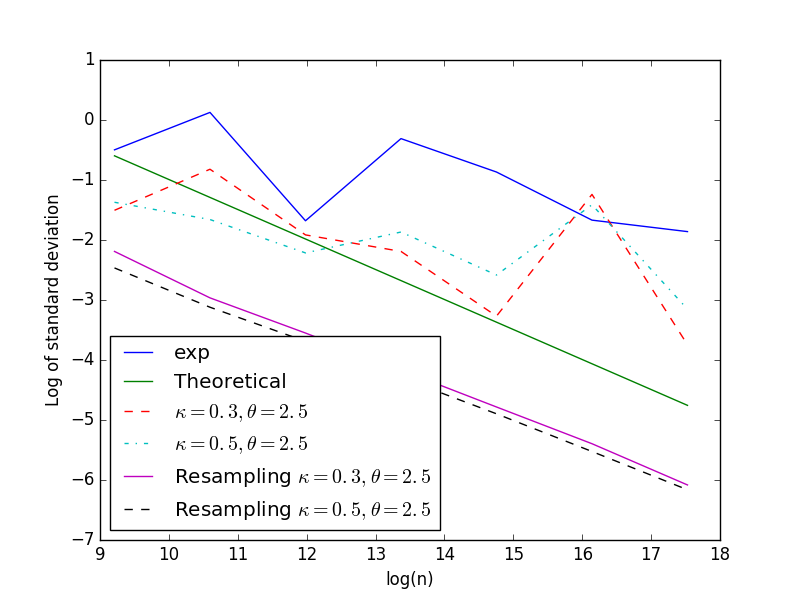}
\caption{Estimation and standard deviation observed for case 4.}
\label{FigCase3}
\end{figure}

\subsection{Some four dimensional cases with $g(x)= (\frac{1}{d} \sum_{i=1}^d x_i-1)^{+}$.}
%The parameters of the SDE \eqref{eq:sde} are   $x \in \R^4$ :
%and $\sigma(t,x)= (0.5 + a \min( (\sum_{i=1}^4 x_i)^2,1)) \mathbb{I}$  for $a$ a positive real number.
The diffusion coefficient is such that $\sigma(t,x)= 0.5 + a \big ( (\sum_{i=1}^4 x_i)^2\wedge 1\big )\mathbb{I}_d$ for any $x\in \R^4$ and for a given positive real, $a$. 
\subsubsection{$a=0.4$}
The reference solution is $0.11806$. We keep $\theta=2.5$ for the gamma distribution.
For this first case, we plot on Figure \ref{FigCase5}   the results obtained with the exponential distribution and the
gamma distribution with and without resampling. 
\begin{figure}[H]
\centering
\includegraphics[width=7.5cm]{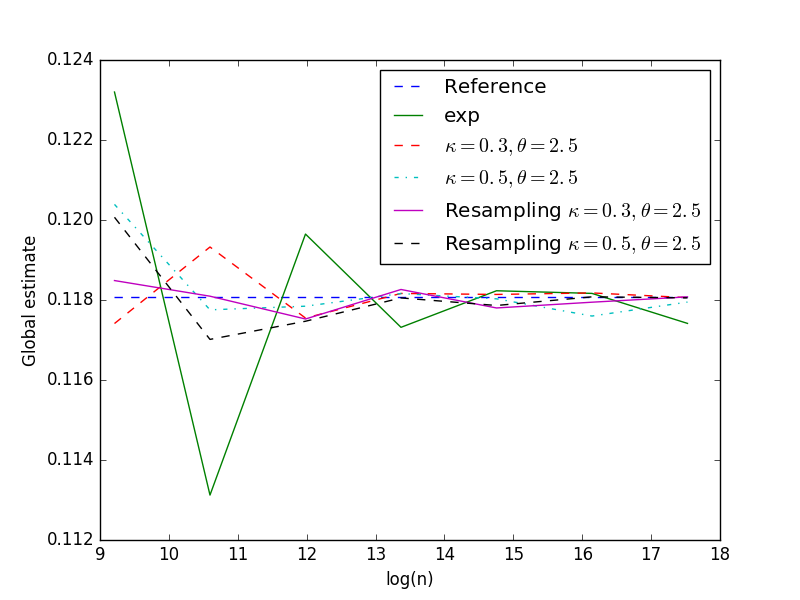}
\includegraphics[width=7.5cm]{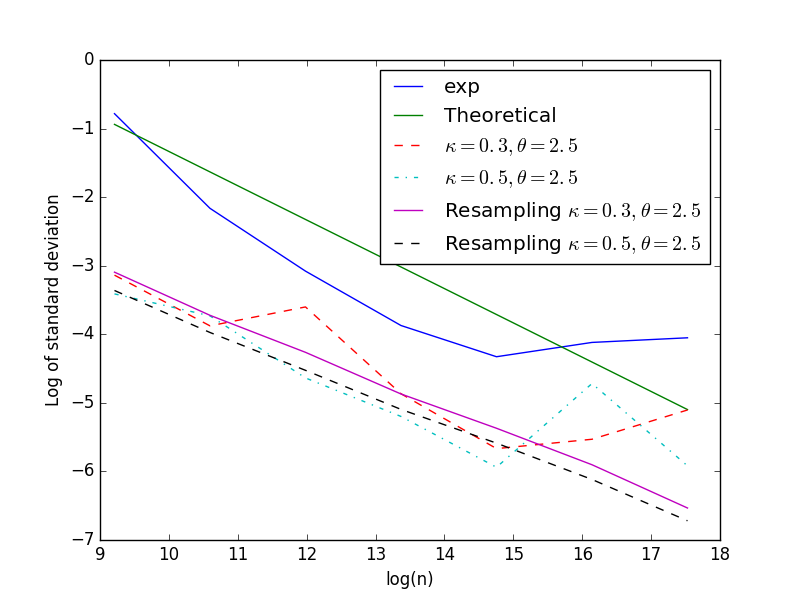}
\caption{Estimation and standard deviation  observed for the first 4 dimensional test case.}
\label{FigCase5}
\end{figure}
On Figure~\ref{FigCase5}, one can observe that for the resampling SMC method with gamma switching times, the log of the standard deviation decreases linearly, whereas without resampling, the decrease of standard deviation is not regular either with gamma or exponential switching times.  On this test case, resampling is effective by reducing the standard deviation by a factor rougly equal to 2 and by stabilizing the results.
\subsubsection{$a=0.6$}
We give the results obtained without resampling on Figure \ref{FigCase6}. The method doesn't seem to converge when a gamma distribution is used without resampling
or when an exponential distribution is used.
\begin{figure}[H]
\centering
\includegraphics[width=7.5cm]{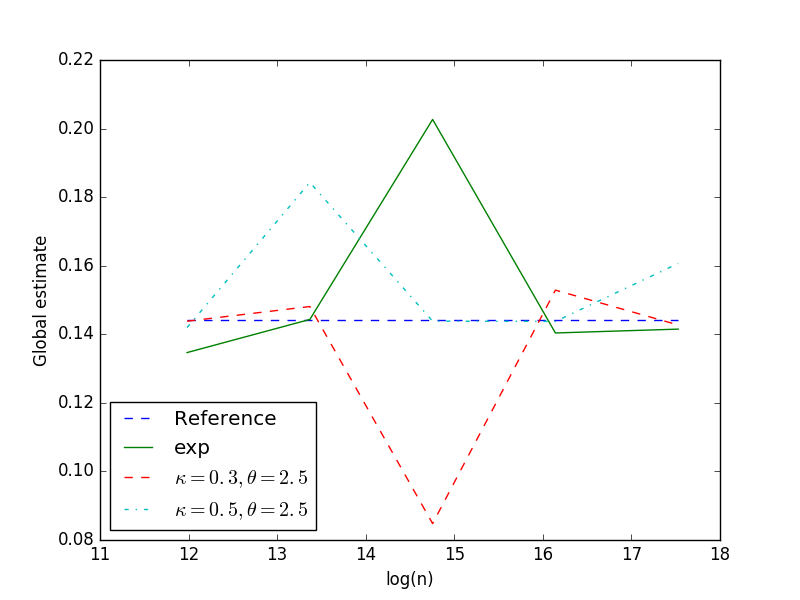}
\includegraphics[width=7.5cm]{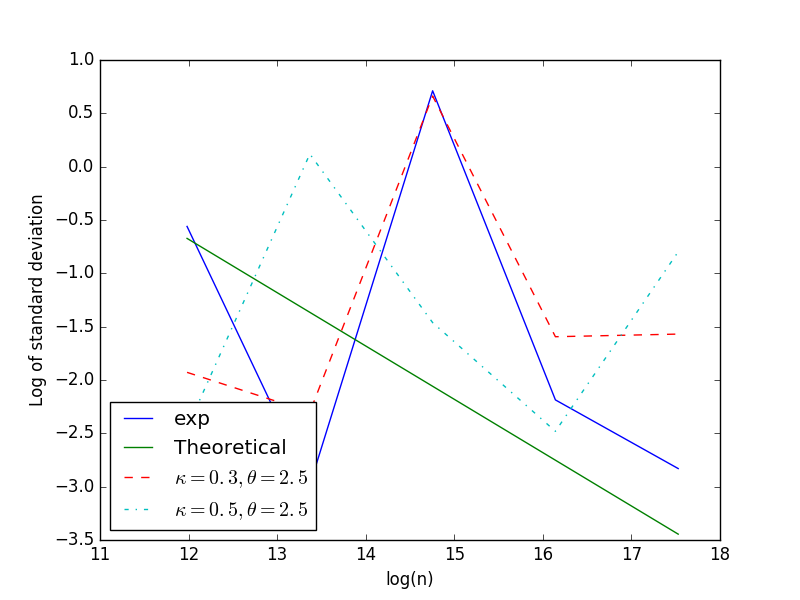}
\caption{Estimation and standard deviation observed for the second 4 dimensional test case for gamma distributions without resampling.}
\label{FigCase6}
\end{figure}

On Figure \ref{FigCase6Zoom}, we only give the results obtained with resampling and the gamma distribution taking different parameters for $\theta$ and $\kappa$. 
We notice that the standard deviation calculated are far higher than in the previous case.
We have difficulties to get the theoretical linear reduction in  the standard deviation. The influence of $\theta$ parameter is not clear on the curves, but because  higher $\theta$ gives  higher jumps, it gives  smaller computational times.
\begin{figure}[H]
\centering
\includegraphics[width=7.5cm]{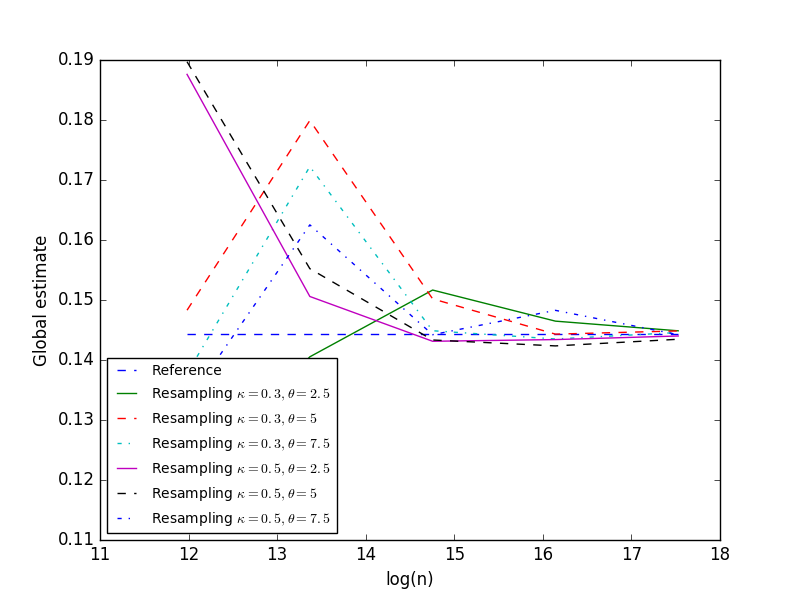}
\includegraphics[width=7.5cm]{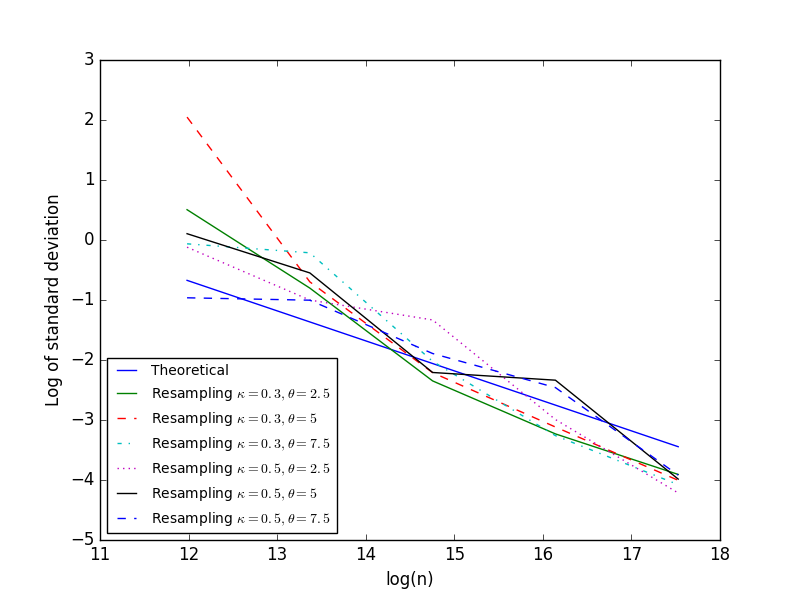}
\caption{Estimation and standard deviation observed for the second 4 dimensional test case for gamma distributions with resampling.}
\label{FigCase6Zoom}
\end{figure}

%-----------------------------------------------------------------------------------
%----------------------------------------------------------------
%\bigskip
\section{Appendix: Usefull notations and proof of Lemma~\ref{lem:gammaN}}
%\hbox{\raisebox{3em}{\vrule depth 0pt height 0.4pt width 17.5cm}}
\label{sec:Appendix}
%----------------------------------------------------------------

\subsection{Classical notations and results}
Let us recall somme classical notations and results stated in~\cite{DelMoral,DelMoral2}. We consider 
%a sequence of Feynman-Kac measures $(\gamma_n)$, associated to 
a Markov chain $(X'_n)$ (with initial probability $\mu_0$ and transition kernel $K_k$) taking values on a sequence of measurable spaces $(E_n, \mathcal{E}_n)$ and a sequence of positive potential functions $(G_n)$ defined on $(E_n, \mathcal{E}_n)$. 
\begin{itemize}
%	\item For all $1\leq p\leq n$, let us define $\Phi_p$, the measure valued function defined on  $ \mathcal{M}(F_p)$, the space of sigma-finite measures of $F_p$, such that for any sigma-finite measure $\mu\in \mathcal{M}(F_p)$, 
%$$\Phi_p(\mu)=(G_{p-1}\cdot \mu)M_p\ .$$
	\item For all $0\leq p\leq n$,  we define the sigma-finite measure $\gamma_p\in \mathcal{M}(E_p)$ such that for any real valued bounded test function $f_p$ defined on $E_p$, 
$$
\gamma_p(f_p)=\E[\,f_p(X'_p)\prod_{0\leq k\leq p-1} G_k(X'_k)\,]\ .
$$
	\item For all $0\leq p\leq n$, let us define the probability measure $\eta_p$  obtained by normalization of $\gamma_p$ and such that for any real valued bounded test function $f_p$ defined on $E_p$, 
$$
\eta_p(f_p)=\frac{\gamma_p(f_p)}{\gamma_p(1)}\ .
$$
Notice that $\gamma_p$ can be written as the following product involving the probability measures $\eta_1,\cdots \eta_p$, 
$$
\gamma_p(f_p)=\eta_p(f_p) \prod_{0\leq k\leq p-1} \eta_k(G_k)\ .
$$
	\item For all $1\leq p\leq n$, let us introduce the nonlinear operator, $\Phi_p$ defined on the space of sigma-finite and non-negative measures $\mathcal{M}^+(E_{p-1})$ and taking values in $\mathcal{M}^+(E_p)$ such that
\begin{equation}
\label{eq:Phi}
\Phi_p(m_{p-1}):=(G_{p-1}\cdot m_{p-1})K_p\ ,\quad \textrm{for any}\quad m_{p-1} \in \mathcal{M}^+(E_{p-1})\ .
\end{equation}
Then the nonlinear evolution of $(\eta_p)$ can be summarized by 
\begin{equation}
\label{eq:phiEvol}
\eta_{p+1}=\Phi_{p+1}(\eta_p)\ ,
\end{equation}

	\item For all $0\leq p\leq n$, let us define the Feynman-Kac semi-group $Q_{p,n}$ associated to the distribution flow $(\gamma_p)_{1\leq p\leq n}$ such that for all $x'_{p}\in E_p$ and any real valued bounded test function $f_n$ defined on $E_n$, 
\begin{equation}
\label{eq:qpn:def}
Q_{p,n}(f_n)(x'_{p})=\E[\,f_n(X'_{n}) \prod_{p\leq k\leq n-1} G_k( X'_{k})\,\vert\, X'_{p}= x'_{p}\,]\ ,
\quad \textrm{with}\ Q_{n,n}=Id\ ,
\end{equation}
%where by convention $G_0= \mathbf{1}$. 
Notice that for all $0\leq p\leq n$,  $\gamma_n$ can be written as the transformation of $\gamma_p$ via $Q_{p,n}$, 
\begin{equation}
\label{eq:gammap}
\gamma_n=\gamma_p Q_{p,n}
%\ , \quad \textrm{where by convention}\ \gamma_0=\mu_0=\mathcal{L}( X'_0)
\ .
\end{equation}
Moreover, for any sigma-finite non-negative measure $\mu\in \mathcal{M}^+(E_p)$, 
\begin{equation}
\label{eq:phip:qp}
\Phi_p(\mu)=\frac{\mu Q_{p-1,p}}{(\mu Q_{p-1,p})(1)} \ .
\end{equation}
	%\item 
\item We introduce $(\gamma_n^N)$ the particle approximation sequence of measures  as defined on Lemma~\ref{lem:gammaN}. 
Notice that by definition~\eqref{eq:gammapN} of $\gamma^N_p$ the following relation holds 
\begin{equation}
\label{eq:etapN:gammapN}
\eta_p^N(f_p)=\frac{\gamma_p^N(f_p)}{\gamma_p^N(1)}\ .
\end{equation}
 \item Let $f_n$ be a real valued test function defined on $E_n$, then the error between $\gamma_n^N(f_n)$ and $\gamma_n(f_n)$ can be decomposed as the sum of $n$ "local errors" as follows, using relation~(\ref{eq:gammap})
\begin{eqnarray*}
(\gamma_n^N-\gamma_n)(f_n)
&=&\sum_{p=1}^n [\gamma_p^N Q_{p,n}-\gamma_{p-1}^N Q_{p-1,n} ](f_n)\quad \textrm{(recalling  that} \ \gamma_0^N:=\gamma_0\textrm{)}\\\\
&=&\sum_{p=1}^n [\gamma_p^N-\gamma_{p-1}^N Q_{p-1,p}]\big (Q_{p,n}(f_n)\big)\ ,
\end{eqnarray*}
Then using relations~(\ref{eq:phip:qp}) and~(\ref{eq:etapN:gammapN}) and recalling that $\gamma_p^N(\mathbf{1})=\gamma_{p-1}^N(G_{p-1})$ yields 
\begin{equation}
\label{eq:decom:gamma}
(\gamma_n^N-\gamma_n)(f_n)
=\sum_{p=1}^n \gamma_{p-1}^N(G_{p-1}) [\eta_p^N-\Phi_p(\eta_{p-1}^N)]\big (Q_{p,n}(f_n)\big)\ .
\end{equation}
%    {\bf XW avec la notation \ref{eq:phip:qp} il manque le terme $\eta_{p-1}^N (Q_{p-1,p})(1)$    }
\item Let us introduce the following notations 
\begin{equation}
\label{eq:DG}
\begin{array}{lll}
\Gamma_p^N(f_p)&:=&(\gamma_{p}^N-\gamma_{p})(f_p)\ ,\\%\quad \textrm{and}\quad 
\Delta_{p}^N(f_p)&:=&[\eta_p^N-\Phi_p(\eta_{p-1}^N)]\big (f_p\big )=\frac{1}{N}\sum_{i=1}^N f_p(\xi_p^i)-\Phi_p(\eta_{p-1}^N)(f_p)\ .
\end{array}
\end{equation}
\item For any $p\geq 0$, let us introduce the $\sigma$-algebra $\mathcal{G}_p$ generated by the particle system until the $p$-th generation, 
observe that since 
$$
\eta_p^N=S^N\big (\Phi_p(\eta_{p-1}^N)\big )=\frac{1}{N}\sum_{i=1}^N \delta _{\xi^i_p}\ ,
$$
where $(\xi^1_p,\cdots ,\xi^N_p)$ are i.i.d. according to $\Phi_p(\eta_{p-1}^N)$ conditionally to $\mathcal{G}_{p-1}$. Thus 
\begin{equation}
\label{eq:biasDG}
\E[\Delta ^N_{p}(f_p)\,\vert\, \mathcal{G}_{p-1}]=\frac{1}{N}\sum_{i=1}^N \E[f_p(\xi^i_p)\vert \mathcal{G}_{p-1}]-\Phi_p(\eta_{p-1}^N)(f_p)=0\ ,
\quad\textrm{and}\quad 
\E[\Gamma_p^N(f_p)]=0\ .
\end{equation}
\item Moreover one can bound the conditional variance of $\Delta ^N_{p}(f_p)$ as follows 
\begin{eqnarray}
\label{eq:DeltaV}
\E[(\Delta ^N_{p}(f_p))^2\,\vert\, \mathcal{G}_{p-1}]&=&\E[\big ((\eta_p^N-\Phi_p(\eta_{p-1}^N)) (f_p)\big)^2\,\vert\, \mathcal{G}_{p-1}] \nonumber \\
&=&\E[\big (\frac{1}{N}\sum_{i=1}^N f_p(\xi^i_p)-\Phi_p(\eta_{p-1}^N)\big)^2\,\vert\, \mathcal{G}_{p-1}] \nonumber \\
&=&
\frac{1}{N}[\Phi_p(\eta_{p-1}^N)) (f_p^2)-\big (\Phi_p(\eta_{p-1}^N)) (f_p)\big )^2] \nonumber\\
&\leq &
\frac{1}{N}\Phi_p(\eta_{p-1}^N)) (f_p^2) \nonumber \\
&=&
\frac{1}{N}(G_{p-1}\cdot \eta_{p-1}^N)(K_p(f_p^2))\ .
\end{eqnarray}
\end{itemize}
 %   {\bf XW Je comprends toujours  pas le passage ligne 1 a ligne 2 dans derniere expression. Tant pis j'admets}  
We are now in a position to prove Lemma~\ref{lem:gammaN}.

\subsection{Proof of Lemma~\ref{lem:gammaN}}

Let us introduce the notation $G_{k,p}:=Q_{k,p}(G_p)$ for any $p\geq k\geq 0$. 
Recalling~\eqref{eq:decom:gamma} and notations~\eqref{eq:DG} gives 
%\begin{eqnarray*}
%\Gamma_p^N(G_p)&=&\sum_{k=1}^p \gamma^N_{k-1}(G_{k-1}) [\eta_k^N-\Phi_k(\eta_{k-1}^N)]\big (Q_{k,p}(G_p)\big)\\
%&=&\gamma^N_{p-1}(G_{p-1}) \Delta_p^N\big (Q_{p,p}(G_p)\big)+\sum_{k=1}^{p-1} \gamma^N_{k-1}(G_{k-1}) \Delta_k^N\big (Q_{k,p-1}(Q_{p-1,p}(G_p))\big)\\
%&=&\gamma^N_{p-1}(G_{p-1})\Delta_{p}^N(G_{p})+\Gamma^N_{p-1}(G_{p-1,p})
%\end{eqnarray*}
%More generally, one can check that for $k=1,\cdots p$ the following relation holds 
%$$
%\Gamma^N_k(G_{k,p})=\gamma^N_{k-1}(G_{k-1})\Delta_{k}^N(G_{k,p})+\Gamma^N_{k-1}(G_{k-1,p})
%$$
\begin{eqnarray*}
\Gamma_k^N(G_{k,p})&=&\sum_{q=1}^k \gamma^N_{q-1}(G_{q-1}) [\eta_q^N-\Phi_q(\eta_{q-1}^N)]\big (Q_{q,k}(G_{k,p})\big)\\
&=&\gamma^N_{k-1}(G_{k-1}) \Delta_k^N\big (Q_{k,k}(G_{k,p})\big)+\sum_{q=1}^{k-1} \gamma^N_{q-1}(G_{q-1}) \Delta_q^N\big (Q_{q,k-1}(Q_{k-1,k}(G_{k,p}))\big)\\
&=&\gamma^N_{k-1}(G_{k-1})\Delta_{k}^N(G_{k,p})+\Gamma^N_{k-1}(G_{k-1,p})\ .
\end{eqnarray*}
Using~\eqref{eq:biasDG} stating that  $\E[\Delta ^N_{k}(G_{k,p})\,\vert\, \mathcal{G}_{k-1}]=0$ gives 
$$
\E[\big (\Gamma^N_k(G_{k,p})\big )^2\,\vert\, \mathcal{G}_{k-1}]
=
\big (\gamma^N_{k-1}(G_{k-1})\big )^2\E[\big (\Delta_{k}^N(G_{k,p})\big )^2\,\vert\, \mathcal{G}_{k-1}]+\big (\Gamma^N_{k-1}(G_{k-1,p})\big )^2\ .
$$
Recalling assumption~\eqref{eq:assG}, observe that for any $ x'_{k-1}\in E_{k-1}$ 
\begin{eqnarray*}
K_k(G^2_{k,p})( x'_{k-1})&=&\E\big [ (\E[ G_{k:p}( X'_{p})\vert X'_{k}])^2\big \vert X'_{k-1}=x'_{k-1}]\\
&\leq &\E\big [ G^2_{k:p}(X'_{p}) \vert X'_{k-1}=x'_{k-1}]\\
&\leq & A^{p-k+1}<\infty \ .
\end{eqnarray*}
Then using the bound~\eqref{eq:DeltaV}, with  $p=k$ and $G_{k,p}$ as a test function implies 
$$
\E[\big (\Delta_{k}^N(G_{k,p})\big )^2\,\vert\, \mathcal{G}_{k-1}]\leq \frac{A^{p-k+1}}{N}\ .
$$
Using the above inequality and recalling that $\gamma^N_{k-1}(G_{k-1})=\gamma_{k-1}(G_{k-1})+\Gamma^N_{k-1}(G_{k-1})$ yields 
$$
\E[\big (\Gamma^N_k(G_{k,p})\big )^2\,\vert\, \mathcal{G}_{k-1}]
\leq 
\frac{A^{p-k+1}}{N} [\big (\gamma_{k-1}(G_{k-1})\big )^2+\big (\Gamma^N_{k-1}(G_{k-1})\big )^2]+\big (\Gamma^N_{k-1}(G_{k-1,p})\big )^2\ .
$$
     %{\bf  \\ XW Je ne vois pas d'ou viens le $(\Gamma^N_{k-1}(G_{k-1})\big )^2]$
     %Il me semble que $\E[\big (\Delta_{k}^N(G_{k,p})\big )^2\,\vert\, \mathcal{G}_{k-1}] \le \frac{C}{N}$.\\
       %}
Again recall that by assumption~\eqref{eq:assG} 
$$
\gamma_{k-1}(G_{k-1}):=\E[G_{1:k-1}(X'_{k-1})]\leq (\E[G^2_{1:k-1}(X'_{k-1})])^{1/2}\leq A^k<\infty \ ,
$$
which finally yields 
$$
\E[\big (\Gamma^N_k(G_{k,p})\big )^2]
\leq 
\frac{A^{p+1}}{N} (1+\E[\big (\Gamma^N_{k-1}(G_{k-1})\big )^2])+\E[\big (\Gamma^N_{k-1}(G_{k-1,p})\big )^2]\ .
$$
Adding the above inequality from $k=1$ to $k=p$ gives for any $p\leq n$
$$
\E[\big (\Gamma^N_p(G_{p})\big )^2]
\leq 
\frac{A^{p+1}}{N} \sum_{k=1}^p(1+\E[\big (\Gamma^N_{k-1}(G_{k-1})\big )^2])\ .
\leq 
\frac{A^{n+1}}{N} \sum_{k=1}^p(1+\E[\big (\Gamma^N_{k-1}(G_{k-1})\big )^2])\ .
$$
We obtain by recursion
\begin{equation}
\label{eq:BoundGammaNp}
\E[\big (\Gamma^N_p(G_{p})\big )^2]
\leq 
(1+\frac{A^{n+1}}{N})^p-1
%\leq p\frac{C}{N} +\sum_{k=2}^\infty \frac{1}{k!}\left (p\frac{C}{N}\right )^k
\ .
\end{equation}
  %{\bf XW
    %Pas vraiment. Si $\E[\big (\Gamma^N_1(G_{1})\big )^2] \le \frac{C}{N}$ alors
    %$\E[\big (\Gamma^N_p(G_{p})\big )^2] \le \frac{C}{N} ((1+\frac{C}{N})^p-1)$ sinon plus lourd il me semble.
    %}
Now let us consider a  test function $f_n$ verifying assumption~\eqref{eq:assf}.  Using again~\eqref{eq:biasDG} stating that  $\E[\Delta ^N_{k}(G_{k,})\,\vert\, \mathcal{G}_{k-1}]=0$ gives 
$$
\E[\big (\Gamma_n^N(f_n)\big )^2]
=
\sum_{p=1}^n \E[\big (\gamma^N_{p-1}(G_{p-1})\big )^2 \big (\Delta_p^N(Q_{p,n}(f_n))\big )^2]\ .
$$
By~\eqref{eq:DeltaV} and Assumption~\eqref{eq:assf}, we obtain $\E[\big (\Delta_p^N(Q_{p,n}(f_n))\big )^2\,\vert\mathcal{G}_{p-1}]\leq B/N$ which yields 
\begin{eqnarray*}
\E[\big (\Gamma_n^N(f_n)\big )^2]
&\leq &
\frac{B}{N}\sum_{p=1}^n \E[\big (\gamma^N_{p-1}(G_{p-1})\big )^2 ]\\
&\leq & 2\frac{B}{N}\sum_{p=1}^n\left (\big (\gamma_{p-1}(G_{p-1})\big )^2 +\E[\big (\Gamma^N_{p-1}(G_{p-1})\big )^2] \right )\\
&\leq & 2\frac{B}{N}\sum_{p=1}^n\left ( A^p +\E[\big (\Gamma^N_{p-1}(G_{p-1})\big )^2] \right )\quad \textrm{since}\ (\gamma_{p-1}(G_{p-1}))^2\leq A^p\ .
\end{eqnarray*}
By~\eqref{eq:BoundGammaNp} we finally get 
$$
\E[\big (\Gamma_n^N(f_n)\big )^2]
\leq  2\frac{B}{N}\sum_{p=1}^n\left ( A^p +(1+\frac{A^{n+1}}{N})^{p-1}-1 \right )
\leq  2\frac{B}{N}\sum_{p=1}^n A^{p+1} \leq  2\frac{B}{N}A^{p+2} \ .
$$
as soon as $N\geq A^{n+1}$ and $A\geq 2$. 
    %{\bf XW Ou a disparu $\E[\big (\Gamma^N_{p-1}(G_{p-1})\big )^2]$ (quel est la majoration?)
      %$\frac{C}{N}\sum_{p=1}^n  \big (\gamma_{p-1}(G_{p-1})\big )^2 \le \frac{C}{N}\sum_{p=1}^n (1+\frac{C}{N})^p-1$  et ..?
      %}

%-----------------------------------------------------------------------------------
%----------------------------------------------------------------
%\bigskip
\section{Appendix: Technicalities related to the proof of Lemma~\ref{lem:Represent}}
%\hbox{\raisebox{3em}{\vrule depth 0pt height 0.4pt width 17.5cm}}
\label{sec:Technics}
%----------------------------------------------------------------
This section provides technical arguments allowing for differentiating under the integral sign that are necessary to prove the second and third identity of~\eqref{eq:Represent}. 

\subsection{Concerning the second identity of~\eqref{eq:Represent}}
Assume the first identity of~\eqref{eq:Represent} is verified. 
Let us introduce the real valued function such that for any $(s,\tilde t, \tilde x, t',x')\in [t',T]\times [0,T]\times \R^d\times [0,T]\times \R^d$  
\begin{equation}
\label{eq:phis}
\phi^{\tilde t, \tilde x}(s,t',x'):=\E [h^{\ast,\tilde t,\tilde x}(s,\tilde X^{\tilde t,\tilde x, t', x'}_s)]\ ,
\end{equation}
where $(\tilde X^{\tilde t,\tilde x, t', x'}_s)$ is the Gaussian process defined by~\eqref{eq:tildeX} and $h^{\ast,\tilde t,\tilde x}$ is the real valued function defined on $[0,T]\times \R^d$ by~\eqref{eq:h}. 
Recalling identity~\eqref{eq:uFK} and using Fubini's lemma gives 
\begin{equation}
\label{eq:uphi}
u(t', x')=\E[g(\tilde X_T^{\tilde t, \tilde x, t' , x'})]+\int_{t'}^T \phi^{\tilde t, \tilde x}(s,t',x')\,ds\ .
\end{equation}
Notice that by a simple application of Elworthy's formula~\cite{fournier1} (which simply results here in the Likelihood ratio of Broadie and Glasserman~\cite{BG96}), we get 
\begin{eqnarray}
\label{eq:Dphi}
D\phi^{\tilde t, \tilde x}(s,t',x')
&=&\E [h^{\ast,\tilde t,\tilde x}(s,\tilde X^{\tilde t,\tilde x, t', x'}_s)\mathcal{M}^{\tilde t,\tilde x}_{t',s}] \nonumber \\
&=&\int_{\R^d} 
\left [ h^{\ast,\tilde t,\tilde x}(s,x'+b(\tilde t,\tilde x)(s-t')+\sqrt{s-t'}\sigma (\tilde t,\tilde x) u) \right .\nonumber \\
&&
\left .\times \frac{(\sigma(\tilde t,\tilde x)^{-1})^\top}{\sqrt{s-t'}} u p(u)\right ]\,du\ ,
\end{eqnarray}
where $p$ denotes the centered and standard Gaussian density on $\R^d$ and $\mathcal{M}^{\tilde t,\tilde x}_{t',s}$ is the Malliavin weight defined at~\eqref{eq:Malliavin}. 
Recall that $b$ and $a$ are Lipschitz w.r.t. the space variable and $1/2$-H\"older continuous w.r.t. the time variable as stated at item 2. and 3. of Assumption~\ref{ass:unique2} and that  $Dv^\ast$ and $D^2v^\ast$ are bounded as stated in  Assumption~\ref{ass:unique1}. 
%Thus there exists a finite constant $C$ (that may change from line to line) such that 
%\begin{eqnarray*}
%\Vert h^{\ast,\tilde t,\tilde x}(s,x'+b(\tilde t,\tilde x)(s-t')+\sqrt{s-t'}\sigma (\tilde t,\tilde x)  u)\Vert 
%&\leq & 
%C \Vert  x'+ b(\tilde t , \tilde x) (s-t')+\sqrt{s-t'}\sigma (\tilde t, \tilde x) u-\tilde x\Vert \\
%&&+C\vert s-\tilde t\vert \\
%&\leq &
%C T\Big (1+\Vert \tilde x-x'\Vert +\Vert b(\tilde t , \tilde x) \Vert \Big )\\
%&&+C\sqrt{T} \Vert \sigma (\tilde t, \tilde x)\Vert  \Vert u\Vert \ .
%\end{eqnarray*}
%Thus 
%\begin{eqnarray}
%\label{eq:Dphib}
%\Vert D \phi^{\tilde t, \tilde x}(s,t',x')\Vert
%&\leq & 
%C \frac{\Vert (\sigma(\tilde t,\tilde x)^{-1} )^T\Vert}{\sqrt{s-t'}}\Big (1+\Vert \tilde x-x'\Vert +\Vert b(\tilde t , \tilde x) \Vert \Big )\int_{\R^d}\Vert u\Vert p(u)du \nonumber \\
%&&+C\frac{\Vert(\sigma(\tilde t,\tilde x)^{-1})^T\Vert}{\sqrt{s-t'}} \Vert \sigma (\tilde t, \tilde x) \Vert \int_{\R^d}\Vert u\Vert^2p(u)du  \nonumber \\
%&\leq &
%C \frac{\Vert (\sigma(\tilde t,\tilde x)^{-1})^T\Vert}{\sqrt{s-t'}}\Big (1 +\Vert b(\tilde t , \tilde x) \Vert \Big ) +C\frac{\Vert (\sigma(\tilde t,\tilde x)^{-1})^T\Vert}{\sqrt{s-t'}} \Vert \sigma (\tilde t, \tilde x) \Vert \ ,
%\end{eqnarray}
%for any $x'$ such that $\Vert x-x'\Vert \leq C$.
%
%{\bf XW : je propose de recourcir en rajoutant les normes \\
    Thus there exists a finite constant $C$ depending on $T$ and that may change from line to line such that  
    \begin{eqnarray*}
      \Vert h^{\ast,\tilde t,\tilde x}(s,x'+b(\tilde t,\tilde x)(s-t')+\sqrt{s-t'}\sigma (\tilde t,\tilde x)  u)\Vert 
      &\leq & 
      C  (1+\Vert \tilde x-x'\Vert +\Vert b(\tilde t , \tilde x) \Vert  +  \Vert \sigma (\tilde t, \tilde x) \Vert   \Vert u\Vert )\ .
    \end{eqnarray*}
    Thus for any $x'$ such that $\Vert \tilde x-x'\Vert \leq C$.
    \begin{eqnarray}
      \label{eq:Dphib}
      \Vert D \phi^{\tilde t, \tilde x}(s,t',x')\Vert 
      &\leq & 
      C \frac{ \Vert \sigma(\tilde t,\tilde x)^{-1} \Vert }{\sqrt{s-t'}}\Big (1 +\Vert b(\tilde t , \tilde x) \Vert  +  \Vert \sigma (\tilde t, \tilde x) \Vert \Big ) ,
    \end{eqnarray}    
%  }
Since the term on the r.h.s of the above inequality is integrable w.r.t $s$ on $[t',T]$ one can differentiate under the integral sign in~\eqref{eq:uphi} %using mean value theorem and assumption \ref{ass:unique} , 
which ends the proof of the second identity of~\eqref{eq:Represent}. 

%%%%%%%%%%%%%%%%%%%%%%%%%%%%%%%%%%%%%%%%%%%%%%%%%%%%%%%
\subsection{Concerning the third identity of~\eqref{eq:Represent}}

Let us introduce the $\R^d$ valued function $\Psi$ such that for any $(\tilde t, \tilde x, t',x')\in  [0,T]\times \R^d\times [0,T]\times \R^d$  
\begin{equation}
\label{eq:Psi}
\Psi(\tilde t, \tilde x, t',x'):=\int_t^T \psi(s,\tilde t,\tilde x, t',x') ds
\end{equation}
where 
$\psi(s,\tilde t,\tilde x, t',x'):=\psi_1(s,\tilde t,\tilde x, t',x')+\psi_2(s,\tilde t, \tilde x, t',x')$ with
\begin{equation}
\label{eq:psi12}
\begin{array}{l}
\psi_1(s,\tilde t,\tilde x, t',x'):=\frac{1}{T-t} \E[g(\tilde X_T^{ \tilde t,\tilde x, t',x'})\mathcal{M}^{\tilde t,\tilde x}_{t',T}]\\ \\
\psi_2(s,\tilde t, \tilde x, t',x'):=\E[h^{\tilde t,\tilde x}(s,\tilde X^{\tilde t,\tilde x, t',x'}_s)\mathcal{M}^{\tilde t,\tilde x}_{t',s}] \ .
\end{array}
\end{equation}
Observe that $\psi=(\psi^1,\cdots ,\psi^d)$ is differentiable w.r.t. the variable $x'$ and for any $j=1,\cdots \,,_, d$ and $i=1,\cdots \,,\,d$
$$
\frac{\partial \psi^j}{\partial x'_i}(s,\tilde t,\tilde x, t',x')=\frac{\partial \psi^j_1}{\partial x'_i}(s,\tilde t,\tilde x, t',x')+\frac{\partial \psi^j_2}{\partial x'_i}(s,\tilde t,\tilde x, t',x')\ ,
$$
where
\begin{equation}
\label{eq:psi12Deriv}
\begin{array}{l}
\frac{\partial \psi^j_1}{\partial x'_i}(s,\tilde t,\tilde x, t',x')=\Big (\frac{1}{T-t}\E[g(\tilde X_T^{ \tilde t, \tilde x, t',x'})\mathcal{V}^{ \tilde t, \tilde x}_{t',T}]\Big )_{i,j}\\ \\
\frac{\partial \psi^j_2}{\partial x'_i}(s,\tilde t,\tilde x, t',x')=\Big (\E[ h^{\ast,\tilde t, \tilde x}(s,X_s^{\tilde t,\tilde x,t',x'})\mathcal{V}^{ \tilde t, \tilde x}_{t',s}]\big )_{i,j}
\end{array}
\end{equation}
Notice that $\Psi$ does not depend on the pair $(\tilde t,\tilde x)$, hence one can fix $\tilde t=t'=t\in [0,T]$. 
Let us consider a fixed point $x\in\R^d$, two indexes  $i,j\in \{1,\cdots d\}$, and a point $\tilde x\in \R^d$ such that each coordinate $\tilde x_\ell=x_\ell$ for any $\ell\neq i$ and $\tilde x_i$ is fixed at a given value. We want to prove that $\frac{\partial \Psi^j}{\partial x'_i}(t,\tilde x, t,x)$ exists and is continuous (which implies the differentiability of $\Psi$) and to give an explicit expression for it. 
By the mean value theorem, there exists a real $\theta^j (t, \tilde x_i, x,h) \in [-1,1]$ (to simplify the notations we will forget the dependence on $t$) such that for any $i=1,\cdots d$
\begin{align}
\label{eq:ref}
\frac{1}{2h}[\Psi^j(t, \tilde x, t,x+he_i)- &\Psi^j(t, \tilde x, t,x-he_i)] \nonumber \\
&=
\int_t^T \frac{1}{2h} 
[\psi^j(s,t, \tilde x, t,x+he_i)-\psi^j(s,t, \tilde x, t,x-he_i)]\,ds \nonumber\\
&=
\int_t^T\frac{\partial \psi^j}{\partial x'}(t,\tilde x, t, x+\theta^j(\tilde x_i, x,h)h e_i)\,ds\ ,
\end{align}
where $e_i$ denotes the vector of $\R^d$ with zeros coordinates except for the $i^{th}$ coordinate that equals $1$. 
Observe that in full generality the real $\theta^j(\tilde x_i, x,h)$ resulting from the mean value theorem depends on $(\tilde x,x,h)$ and not only on $(\tilde x_i,x,h)$. However, since we consider the specific situation where $\tilde x_j=x_j$ for all $j\neq i$ one can express this real as a function of  $(\tilde x_i, x,h)$.  
Consider the following equation w.r.t. the variable $\tilde x_i \in \R$ 
$$
\lambda ^{i,j} _{x,h}(\tilde x_i)=0
\ ,\quad \textrm{where} \quad \lambda ^{i,j}_{x,h}(\tilde x_i):=x_i+\theta ^j(\tilde x_i,x,h)h-\tilde x_i \ .
$$
One can check that there exists a solution $\hat x_i(x,h)$ to this equation. Indeed, taking $\tilde x_i=x_i+h$ and $\tilde x_i=x_i-h$ and recalling that $\theta ^j(\tilde x_i,x,h)\in [-1,1]$, we obtain 
$$
\lambda^{i,j} _{x,h}(x_i+h)=h(\theta ^j(x_i+h,x,h)-1)\leq 0\ ,\quad \textrm{and}\quad \lambda^{i,j} _{x,h}(x_i-h)=h(\theta ^j(x_i-h,x,h)+1)\geq 0
$$
which, by continuity of $\lambda^{i,j} _{x,h}$, implies the existence of a solution. 
Now we choose to take in equation~\eqref{eq:ref}, $\tilde x$ as the vector having the same coordinates as $x$ except that $\tilde x_i=\hat x_i(x,h)$, this vector will be denoted by $\hat x(x,h)$.  Observe that by construction choosing $\tilde x=\hat x(x,h)$ implies
$$
\tilde x=x+\theta^j(\tilde x_i, x,h)h e_i\ .
$$
We are now interested in the limit of~\eqref{eq:ref} as $h\rightarrow 0$. 
The technical point will consist in applying Lebesgue theorem to permute the limit with the integral sign. 
First, by Lebesgue theorem, and observing that $\hat x_i(x,h)\rightarrow x_i$ when $h\rightarrow 0$, we have 
\begin{eqnarray*}
 \lim_{h\rightarrow 0} 
\int_t^T\frac{\partial \psi ^j_1}{\partial x'_i}(s,t,\hat{x}(x,h), t, \hat x(x,h))\,ds&=&
\int_t^T\lim_{h\rightarrow 0} 
\frac{\partial \psi ^j_1}{\partial x'_i}(s,t,\hat{x}(x,h), t, \hat x(x,h))\,ds\\
&=&
\int_t^T \frac{\partial \psi ^j_1}{\partial x'_i}(s,t,x, t, x)\,ds\ .
\end{eqnarray*}
Considering the integral term involving $\psi^j_2$, using the Lipschitz and $\alpha$-H\"older properties of $b$ and $\sigma$, we get
 $$ | h^{\ast,t, \hat x(x,h)}(s,\tilde X^{t,\hat x(x,h), t, \hat x(x,h)}_s )| \le  C(\hat x(x,h),t) \big [(s-t)^\alpha +\vert W_t-W_s\vert\big ]
$$
 so 
$$ \vert \frac{\partial \psi^j_2}{\partial x'} (s,t,\hat{x}(x,h), t, \hat x(x,h)) \vert \le \frac{C(\hat x(x,h),t)}{(s-t)^{1-\alpha \wedge 1/2}}$$  
where $C(\hat x(x,h),t)$ is locally  bounded due to the non degeneracy hypothesis  and the Lipschitz properties in assumption \ref{ass:unique2}.
 The rhs of the previous equation is  integrable  so that we can use the Lebesgues Theorem, 
\begin{eqnarray*}
 \lim_{h\rightarrow 0} 
\int_t^T\frac{\partial \psi ^j_2}{\partial x'_i}(s,t,\hat{x}(x,h), t, \hat x(x,h))\,ds&=&
\int_t^T\lim_{h\rightarrow 0} 
\frac{\partial \psi ^j_2}{\partial x'_i}(s,t,\hat{x}(x,h), t, \hat x(x,h))\,ds\\
&=&
\int_t^T \frac{\partial \psi ^j_2}{\partial x'_i}(s,t,x, t, x)\,ds\ .
\end{eqnarray*}
We finally obtain 
$$
\frac{\partial \Psi^j}{\partial x'_i}(\tilde t,\tilde x, t,x)=\Big (\E[g(\tilde X_T^{ t, x, t,x})\mathcal{V}^{  t, x}_{t,T}]\Big )_{i,j}+\int_t^T\Big (\E[ h^{\ast,t,  x}(s,X_s^{ t, x,t,x})\mathcal{V}^{ t,  x}_{t,s}]\big )_{i,j} \, ds
$$
which ends the proof.

%%%%%%%%%%%%%%%%%%%%%%%%%%%%%%%%%%%%%%%%%%%%%%%%%%%%%%%
\section*{AKNOWLEDGMENTS}
The authors are very grateful to the anonymous Referees for their careful reading of the paper and the suggestions which have largely contributed to improve the first submitted version.

\bibliographystyle{plain}
\bibliography{Branchement}

\end{document}